\documentclass[a4paper,11pt]{article}
\usepackage{authblk}

\usepackage[utf8]{inputenc}
\usepackage[english]{babel}
\usepackage{fullpage}
\usepackage{xspace}
\usepackage{amsmath}
\usepackage{amsthm}
\usepackage{amssymb}

\usepackage{dsfont}
\usepackage{tikz-cd}
\usepackage[shortlabels]{enumitem}

\usepackage[bookmarks=true,
            bookmarksnumbered=true,
            pagebackref=true,
            colorlinks=true,
            linkcolor=blue,
            citecolor=red]{hyperref}

\title{Regressive versions of Hindman's Theorem}

\author[1]{Lorenzo Carlucci}
\author[2]{Leonardo Mainardi}
\affil[1]{Sapienza University of Rome, Department of Mathematics}
\affil[2]{Sapienza University of Rome, Department of Computer Science}
\affil[1]{lorenzo.carlucci@uniroma1.it}
\affil[2]{leonardo.mainardi@uniroma1.it}

\date{\today{}}

\hypersetup{
  pdfauthor={Lorenzo Carlucci}
  pdftitle ={}
}

\newtheorem{definition}{Definition}
\newtheorem{proposition}{Proposition}
\newtheorem{theorem}{Theorem}
\newtheorem{corollary}{Corollary}
\newtheorem{lemma}{Lemma}

\newtheorem{claim}{Claim}
\newtheorem{open.problem}{Open Problem}
\newtheorem{question}{Question}
\newtheorem{remark}{Remark}

\newtheorem*{theorem*}{Theorem}
\newtheorem*{corollary*}{Corollary}
\newtheorem*{proposition*}{Proposition*}
\newtheorem*{lemma*}{Lemma}
\newtheorem*{fact*}{Fact}
\newtheorem*{claim*}{Claim}
\newtheorem*{open.problem*}{Open Problem}
\newtheorem*{remark*}{Remark}
\newtheorem*{example*}{Example}
\newtheorem*{exercise*}{Exercise}

\newcommand\Nat{\mathbf{N}}
\newcommand\RCA{\mathsf{RCA}}
\newcommand\ACA{\mathsf{ACA}}
\newcommand\WKL{\mathsf{WKL}}

\newcommand\HT{\mathsf{HT}}
\newcommand\IPT{\mathsf{IPT}}

\newcommand\RT{\mathsf{RT}}

\newcommand\regHT{\mathsf{regHT}}
\newcommand\canHT{\mathsf{canHT}}
\newcommand\REG{\mathsf{regRT}}
\newcommand\CAN{\mathsf{canRT}}

\newcommand\FIN{\mathrm{FIN}}
\newcommand\FS{\mathrm{FS}}

\newcommand\compred{\mathrm{c}}

\newcommand\W{\mathrm{W}}
\newcommand\sW{\mathrm{sW}}

\newcommand\ap{\mathsf{ap}}

\newcommand\WO{\mathrm{WO}}

\newcommand\ordX{\mathcal{X}}

\newcommand\WOP{\mathsf{WOP}}
\newcommand\RAN{\mathsf{RAN}}

\newcommand\Q{\mathsf{Q}}
\newcommand\PP{\mathsf{P}}

\newcommand\lambdam{\lambda'}

\begin{document} 

\maketitle

\begin{abstract}

When the Canonical Ramsey's Theorem by Erd\H{o}s and Rado
is applied to regressive functions, one obtains the Regressive Ramsey's Theorem by Kanamori and McAloon.
Taylor proved a ``canonical'' version of Hindman's Theorem, analogous to the Canonical Ramsey's Theorem. 
We introduce the restriction of Taylor's Canonical Hindman's Theorem 
to a subclass of the regressive functions, the $\lambda$-regressive functions, relative to an adequate 
version of min-homogeneity and prove some results about the Reverse Mathematics of this Regressive
Hindman's Theorem and of natural restrictions of it.

In particular we prove that the first non-trivial restriction of the principle is equivalent to Arithmetical 
Comprehension. We furthermore prove that the well-ordering-preservation principle for base-$\omega$ exponentiation is reducible to this same principle by a uniform computable reduction.
\end{abstract}



\section{Introduction and motivation}

Hindman's well-known Finite Sums Theorem \cite{Hin:74} states that for any finite colouring of 
the natural numbers there exists an infinite
subset of positive natural numbers such that all finite sums of distinct terms from that subset get the same
colour. 

The strength of Hindman's Theorem is a major open problem in Reverse Mathematics 
(see, e.g., \cite{Mon:11:open}) since the seminal work of Blass, Hirst and Simpson \cite{Bla-Hir-Sim:87}. They 
showed that Hindman's Theorem is provable in the system $\ACA_0^+$ (axiomatized by closure under the $\omega$-th Turing Jump) and in turn implies $\ACA_0$ (axiomatized by closure under the Turing Jump) over the base system $\RCA_0$. This leaves a huge gap between the upper and the lower bound (we refer to \cite{Sim:SOSOA, Hir:STT:14} and to the recent \cite{Dza-Mum:22} for Reverse Mathematics fundamentals).

Recently, substantial interest has been given to various restrictions of Hindman's Theorem (see \cite{Car:2021} and \cite{Dza-Mum:22} Section 9.9.3 for an overview and references). Dzhafarov, Jockusch, Solomon and Westrick \cite{DJSW:16} proved that the $\ACA_0$ lower bound on Hindman's Theorem already applies to the restriction of the theorem to colourings in $3$ colours and sums of at most $3$ terms (denoted $\HT^{\leq 3}_3$) and that Hindman's Theorem restricted to colourings in $2$ colours and sums of at most $2$ terms (denoted $\HT^{\leq 2}_2$) is unprovable in $\RCA_0$.
The first author in joint work with Ko{\l}odziejczyk, Lepore and Zdanowski later showed that $\HT^{\leq 2}_4$ implies $\ACA_0$ and
that $\HT^{\leq 2}_2$ is unprovable in $\WKL_0$ \cite{Car-Kol-Lep-Zda:20}.
However, no upper bound other than the one for the full Hindman's Theorem is known for $\HT^{\leq 2}_k$, let alone $\HT^{\leq 3}_k$, for any $k > 1$. Indeed, it is an open question in Combinatorics whether Hindman's Theorem for sums of at most $2$ terms is already equivalent to the full Hindman's Theorem (see \cite{Hin-Lea-Str:03}, Question 12).
On the other hand some restrictions of Hindman's Theorem that are equivalent to $\ACA_0$ have been isolated and called ``weak yet strong'' principles by the first author (see \cite{Car:16:wys}). Theorem 3.3 in \cite{Car-Kol-Lep-Zda:20} shows that Hindman's Theorem restricted to colourings in $2$ colours and sums of exactly $3$ terms with an apartness condition on the solution set is a weak yet strong principle in this sense.

In this paper we isolate a new natural variant of Hindman's Theorem, called the {\em Regressive Hindman's Theorem}, modeled on 
Kanamori-McAloon's Regressive Ramsey's Theorem \cite{Kan-McA:87}, and we investigate its strength in terms of provability over $\RCA_0$ and in terms of computable reductions. In particular we prove that the weakest non-trivial restriction of the Regressive Hindman's Theorem is a weak yet strong principle in the sense of \cite{Car:16:wys}, being equivalent to $\ACA_0$. We also show that the Range Existence Principle for injective functions is reducible to that same Regressive Hindman's Theorem by a uniform computable reduction (called Weihrauch reduction). Moreover, we show that the same is true of the Well-Ordering Preservation Principle for base-$\omega$ exponentiation. This principle states that, for any linear order $\ordX$, if $\ordX$ is well-ordered then $\omega^\ordX$ is well-ordered. It is known to be equivalent to $\ACA_0$ (see \cite{Hir:94}); well-ordering principles have received substantial attention in later years (see the recent survey by Michael Rathjen \cite{Rat:wop:pre} for an overview and references). No direct connection to Hindman-type theorems
has been drawn in previous works.

The theorems studied in this paper are $\Pi^1_2$-principles, i.e., principles that can be written in the following form: 
$$ \forall X (I(X) \to \exists Y S(X,Y))$$
where $I(X)$ and $S(X,Y)$ are arithmetical formulas and $X$ and $Y$ are set variables. For principles $\PP$ of this form we call any $X$ that satisfies
$I$ an {\em instance} of $\PP$ and any $Y$ that satisfies $S(X,Y)$ a {\em solution to} $\PP$ {\em for} $X$.
We will use the following notions of computable reducibility between two $\Pi^1_2$-principles $\PP$ and $\Q$, 
which have become of central interest in Computability Theory and Reverse Mathematics in recent years
(see \cite{Dza-Mum:22} for background and motivation).

\begin{enumerate}
\item $\Q$ is {\em strongly Weihrauch reducible} to $\PP$ (denoted $\Q \leq_{\sW} \PP$) if there exist Turing functionals $\Phi$ and $\Psi$ such 
that for every instance $X$ of $\Q$ we have that $\Phi(X)$ is an instance of $\PP$, and if $\hat{Y}$ is a solution to $\PP$ for $\Phi(X)$
then $\Psi(\hat{Y})$ is a solution to $\Q$ for $X$.
\item $\Q$ is {\em Weihrauch reducible} to $\PP$ (denoted $\Q \leq_\W \PP$) if there exist Turing functionals $\Phi$ and $\Psi$ such 
that for every instance $X$ of $\Q$ we have that $\Phi(X)$ is an instance of $\PP$, and if $\hat{Y}$ is a solution to $\PP$ for $\Phi(X)$
then $\Psi(X\oplus \hat{Y})$ is a solution to $\Q$ for $X$.
\item $\Q$ is {\em computably reducible} to $\PP$ (denoted $\Q \leq_{\mathrm{c}} \PP$) if every instance $X$ of $\Q$ computes an instance $\hat{X}$ of $\PP$ such that
if $\hat{Y}$ is any solution to $\PP$ for $\hat{X}$, then there is a solution $Y$ to $\Q$ for $X$ computable from $X \oplus \hat{Y}$. 
\end{enumerate}

The above reducibility notions are related by the following strict implications: 
$$ \leq_\sW \;\Longrightarrow\; \leq_\W \;\Longrightarrow\; \leq_{\mathrm{c}},$$
and make it possible to illuminate subtle differences in the intuitive idea of solving a problem $\Q$ algorithmically from a problem
$\PP$. Note that $\Q \leq_{\mathrm{c}}\PP$ implies that each $\omega$-model of $\RCA_0 + \PP$ is also a model of $\Q$
(the latter fact is usually denoted by $\Q\leq_\omega\PP$). We refer the reader to \cite{Dza-Mum:22} for examples 
witnessing how the three reducibility notions differ.

In the present paper we only establish positive reducibility results, indicating when implications of type $\PP\to \Q$ over
$\RCA_0$ are witnessed by strongly Weihrauch, Weihrauch or computable reductions. A few non-reducibility results are obtained as simple corollaries of our reducibility results and non-reducibility results from the literature. 

\section{Canonical and Regressive Ramsey's Theorems} 

We review some definitions and known facts concerning Ramsey's Theorem and its canonical and regressive versions. 
We use $\Nat$ for the set of natural numbers and $\Nat^+$ for the 
set of positive integers. For $X\subseteq \Nat$ and $n\geq 1$ we denote by $[X]^n$ the set of 
subsets of $X$ of cardinality $n$. For $k\in\Nat^+$ we identify $k$ with $\{0, 1, \dots, k-1\}$. Accordingly, 
for $S\subseteq \Nat$, $c:[S]^n \to k$ indicates a colouring of $[S]^n$ in $k$ colours. Intervals 
are intervals in $\Nat$. We start by 
recalling the statement of the standard countable Ramsey's Theorem.

\begin{definition}[Ramsey's Theorem]
Let $n, k\in\Nat^+$. We denote by $\RT^n_k$ the following principle. For all $c:[\Nat]^n \to k$ there exists
an infinite set $H\subseteq \Nat$ such that $c$ is constant on $[H]^n$. The set $H$ is called homogeneous
or monochromatic for $c$. Also, we use $\RT^n$ to denote $(\forall k \geq 1)\,\RT^n_k$ and $\RT$ to denote $(\forall n \geq 1)\,\RT^n$.
\end{definition}

For $n\in\Nat^+$, $S \subseteq \{1, \dots, n\}$, $I = \{i_1 < \dots < i_n\}\subseteq\Nat$ and $J = \{ j_1 < \dots < j_n\}\subseteq\Nat$ we say 
that $I$ and $J$ {\em agree} on $S$ if and only if for all $s \in S$, $i_s = j_s$. Note that if $S$ is empty then all
$n$-sized subsets of $\Nat$ agree on $S$.

The following generalization of Ramsey's Theorem to colourings in possibly infinitely many colours was established by Erd\H{o}s and Rado \cite{Erd-Rad:50}. 

\begin{definition}[Erd\H{o}s and Rado's Canonical Ramsey's Theorem]\label{defi:can}
Let $n\in\Nat^+$. We denote by $\CAN^n$ the following principle. For all $c: [\Nat]^n \to \Nat$ there exists an infinite set $H\subseteq \Nat$ and a finite (possibily empty) set $S \subseteq \{1, \dots, n\}$ such that for all $I,J \in [H]^n$ the equality $c(I) = c(J)$ holds if and only if $I$ and $J$ agree on $S$.
The set $H$ is called {\em canonical} for $c$. We use $\CAN$ to denote $(\forall n \geq 1)\CAN^n$.
\end{definition}

The Reverse Mathematics of $\CAN^n$ is studied in~\cite{Mil:08}, where it is denoted by $\mathsf{CAN}^n$. 
																		
As observed in~\cite{Mil:08} (Proposition 8.5), $\CAN^1$ is equivalent to $\RT^1$ over $\RCA_0$.

Kanamori and McAloon \cite{Kan-McA:87} isolated a straightforward corollary of the Canonical Ramsey's Theorem 
inspired by Fodor's Lemma in Uncountable Combinatorics. To state the Kanamori-McAloon's principle we need the following definitions.

\begin{definition}[Regressive function]
Let $n\in\Nat^+$. 
A function $c:[\Nat]^n\to\Nat$ is called regressive if and only if, for all $I \in [\Nat]^n$, $c(I) < \min(I)$ if
$\min(I)>0$, else $c(I)=0$. 
\end{definition}

\begin{definition}[Min-homogeneity]\label{defi:minhomo}
Let $n\in\Nat^+$, $c:[\Nat]^n\to \Nat$ and $H\subseteq \Nat$ an infinite set. 
The set $H$ is min-homogeneous for $c$ if and only if the following condition holds:
for any $I, J \in [H]^n$, if $\min(I)=\min(J)$ then $c(I)  = c(J).$
\end{definition}

\begin{definition}[Kanamori-McAloon's Regressive Ramsey's Theorem]
Let $n\in\Nat^+$. We denote by $\REG^n$ the following principle. For all regressive $c:[\Nat]^n \to \Nat$ 
there exists an infinite min-homogeneous set
$H\subseteq\Nat$. We denote by $\REG$ the principle $(\forall n \geq 1) \REG^n$.
\end{definition}

The Reverse Mathematics of $\REG^n$ is studied in~\cite{Mil:08}, where it is denoted
by $\mathsf{REG}^n$. Note that $\REG^1$ is trivial. 
A finite first-order miniaturization of $\REG$ was proved by Kanamori and McAloon \cite{Kan-McA:87} to be independent
from Peano Arithmetic and is often considered as one of the most mathematically natural 
examples of statements independent from that system. 

The following theorem summarizes the main known results about the Reverse Mathematics of the Canonical and Regressive versions of Ramsey's Theorem.

\begin{theorem}\label{thm:mileti}
The following are equivalent over $\RCA_0$.
\begin{enumerate}
\item $\ACA_0$.
\item $\CAN^n$, for any fixed $n \geq 2$.
\item $\REG^n$, for any fixed $n \geq 2$.
\item $\RT^n_k$, for any fixed $n \geq 3$ and $k\geq 2$.
\end{enumerate}
\end{theorem}

\begin{proof}
The equivalences concerning $\ACA_0$ and Ramsey's Theorems are all due to Simpson 
(Theorem III.7.6 in \cite{Sim:SOSOA}), based on the 
computability-theoretic analysis by Jockusch \cite{Joc:72}. 
The fact that $\REG^n$ implies $\ACA_0$ is due to Hirst, see~\cite{Hir:phd}.  
That $\ACA_0$ implies $\CAN^n$ is due to Mileti, using a new proof of the Canonical Ramsey's Theorem~\cite{Mil:08}. 
The implications from $\CAN^n$ to $\RT^n$ and $\CAN^n$ to $\REG^n$ are simple observations. 
\end{proof}

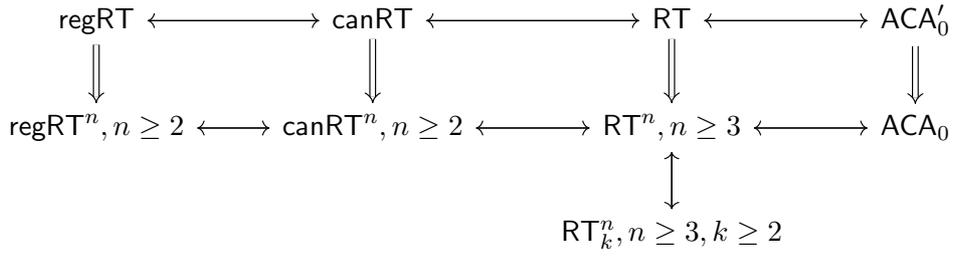
\begin{figure}
\centering
\[\begin{tikzcd}[row sep=2em,column sep=2.5em]
\REG \arrow[d, Rightarrow] \arrow[r] &  \arrow[l] \CAN  \arrow[d, Rightarrow] \arrow[r] & \arrow[l] \RT \arrow[r] \arrow[d, Rightarrow] & \arrow[l] \ACA_0' \arrow[d, Rightarrow]\\
\REG^n, n\geq 2  \arrow[r] & \arrow[l] \CAN^n, n\geq 2 \arrow[r] &  \arrow[l] \RT^n, n\geq 3  \arrow[d] \arrow[r] & \ACA_0 \arrow[l]\\
& & \arrow[u] \RT^n_k, n\geq 3, k\geq 2 &
\end{tikzcd}\]
\caption{Implications over $\RCA_0$. Double arrows indicate strict implications. The equivalences with $\ACA_0$ are from 
Theorem~\ref{thm:mileti}. For the other implications we refer the reader to~\cite{Mil:08}.} \label{fig:implicationsREG}
\end{figure}

Theorem 6.14 in Hirst's Ph.D. Thesis~\cite{Hir:phd} gives an implication (and a strong Weihrauch reduction) 
from $\RT^{2n-1}_2$ to $\REG^n$, for all $n \geq 2$. 

There seems to be no direct and exponent-preserving proof of $\RT^n$ from $\REG^n$ in the literature. 
A simple proof of this implication is in Proposition \ref{prop:regrt} below. As pointed out by one of the anonymous reviewers of the present paper, a simple forgetful function argument proves $\RT^n$ from $\REG^{n+1}$.

Also note that Ramsey's Theorem for pairs is strictly between $\RCA_0$ and $\ACA_0$ (see~\cite{Hir:STT:14} for details). 
Moreover, the principles $\CAN$, $\RT$ and $\REG$ are all equivalent to $\ACA_0'$, the system obtained by adding to 
$\RCA_0$ the axiom $\forall n \forall X \exists Y (Y = (X)^{(n)})$ stating the closure of the set universe under the $n$-th Turing Jump for every $n$; see~\cite{Mil:08}, Proposition~8.4. The main relations among Canonical, Regressive and standard Ramsey's Theorems with respect to implication over $\RCA_0$ are visualized in Figure~\ref{fig:implicationsREG}.

\section{Canonical and Regressive Hindman's Theorems} 

We start by recalling Hindman's Finite Sums Theorem~\cite{Hin:74}. For a set $X\subseteq \Nat$ we denote by $\FS(X)$ the 
set of all finite non-empty sums of distinct elements of $X$.

\begin{definition}[Hindman's Theorem]\label{thm:ht}
Let $k\in\Nat^+$. We denote by $\HT_k$ the following principle. For all $c: \Nat \to k$ there exists an infinite set 
$H\subseteq \Nat$ such that $c$ is constant on $\FS(H)$. We denote by $\HT$ the principle $(\forall k\geq 1)\,\HT_k$.
\end{definition}

For technical convenience, Hindman's Theorem is usually stated with $\Nat^+$ instead of $\Nat$. Obviously we can always assume without loss of generality that $H$ in the above definition is a subset of $\Nat^+$.

Taylor \cite{Tay:76} proved the following ``canonical'' version of Hindman's Theorem, analogous to the Canonical Ramsey's Theorem by Erd\H{o}s and Rado (Definition \ref{defi:can}). We denote by $\FIN(\Nat)$ the set of non-empty finite subsets of $\Nat$.

\begin{definition}[Taylor's Canonical Hindman's Theorem]\label{thm:taylor}
We denote by $\canHT$ the following principle. For all $c: \Nat \to \Nat$ there exists an infinite set $H=\{h_0 < h_1 < \cdots \} \subseteq \Nat$ such that one of the following holds:
\begin{enumerate}
\item For all $I, J \in \FIN(\Nat)$, $c(\sum_{i \in I} h_i) = c(\sum_{j\in J} h_j)$.
\item For all $I, J \in \FIN(\Nat)$,  $c(\sum_{i \in I} h_i) = c(\sum_{j\in J} h_j)$ if and only if $I=J$.
\item For all $I, J \in \FIN(\Nat)$,  $c(\sum_{i \in I} h_i) = c(\sum_{j\in J} h_j)$ if and only if $\min(I)=\min(J)$.
\item For all $I, J \in \FIN(\Nat)$,  $c(\sum_{i \in I} h_i) = c(\sum_{j\in J} h_j)$ if and only if $\max(I)=\max(J)$.
\item For all $I, J \in \FIN(\Nat)$,  $c(\sum_{i \in I} h_i) = c(\sum_{j\in J} h_j)$ if and only if $\min(I)=\min(J)$ and $\max(I)=\max(J)$.
\end{enumerate}
The set $H$ is called {\em canonical} for $c$.
\end{definition}

None of the cases in Definition~\ref{thm:taylor} can be omitted without falsifying Taylor's Theorem. 
For technical convenience, $\canHT$ is usually stated with $\Nat^+$ instead of $\Nat$. We can always assume without loss of generality that $H$ in the above definition is a subset of $\Nat^+$.

We first observe how Taylor's Theorem implies the standard Hindman's Theorem just as the Canonical Ramsey's Theorem implies
Ramsey's Theorem. 

\begin{proposition}\label{prop:can_to_ht}
$\canHT$ implies $\HT$ over $\RCA_0$. Moreover, $\canHT \geq_\sW \HT$.
\end{proposition}

\begin{proof}
Let $c:\Nat\to k$ be a finite colouring of $\Nat$, with $k\in\Nat^+$. By $\canHT$ there exists an infinite
set $H\subseteq \Nat^+$ such that one of the five canonical cases in Definition~\ref{thm:taylor} occurs.
It is easy to see that, since $c$ is a colouring in $k$ colours, only case (1) of Definition~\ref{thm:taylor} can occur. Thus $\FS(H)$ is
homogeneous for $c$. The argument obviously establishes a strong Weihrauch reduction. 
\end{proof}

In the usual Finite Unions versions of Hindman's Theorem and of Taylor's Theorem the instance is a finite colouring of the finite subsets of $\Nat$ and the solution is an infinite sequence $(B_i)_{i\in\Nat}$ of finite subsets of $\Nat^+$ satisfying the so-called {\em block condition}: for all $i <j$, $\max(B_i) < \min(B_j)$; henceforth we will write $X < Y$ to indicate that this condition holds for the finite sets $X$ and $Y$. When this condition is dropped, Hindman's Finite Unions Theorem becomes much weaker (in particular, provable in $\RCA_0$) as shown by Hirst (see~\cite{Car:2021} for references).
We introduce the corresponding property in the finite sums setting. This property is already implicit in Hindman's original proof \cite{Hin:74} and was called {\em apartness} by the first author in~\cite{Car:16:wys}. Let $n\in \Nat^+$. If $n=2^{t_1}+\dots+2^{t_p}$ with $0 \leq t_1 < \dots < t_p$ let $\lambda(n)=t_1$ and $\mu(n)=t_p$ (the notation is from~\cite{Bla-Hir-Sim:87}). We set $\lambda(0)=\mu(0)=0$.

\begin{definition}[Apartness Condition] 
A set $X$ satisfies the apartness condition if for all $x,x'\in X$ such that $x < x'$, we have $\mu(x)<\lambda(x')$. If $X$ satisfies the apartness condition we say that $X$ is apart.
\end{definition}

If $\mathsf{P}$ is a Hindman-type principle, we denote by $\mathsf{P}$ {\em with apartness} or $\mathsf{P}[\ap]$, the principle $\mathsf{P}$ with the apartness condition imposed on the solution set.

In Hindman's original proof the apartness condition is ensured 
by a simple counting argument (Lemma 2.2 in~\cite{Hin:72}) 
on any solution to the Finite Sums Theorem, i.e., an infinite $H\subseteq\Nat$ 
such that $\FS(H)$ is monochromatic (Lemma 2.3 in~\cite{Hin:72}). 
As noted in \cite{Bla-Hir-Sim:87}, the proof shows that a solution satisfying the apartness condition 
can be obtained computably in any such solution. In the Reverse Mathematics setting, one needs to be slightly more
careful to establish that $\HT$ 
implies $\HT$ {\em with apartness} over $\RCA_0$. 

We first check that Lemma 2.2 in \cite{Hin:72} holds in $\RCA_0$.

\begin{lemma}\label{lem:counting} The following is provable in $\RCA_0$: 
For all $\ell$, for all $k$, for all finite sets $X$, if $X$ has cardinality $2^k$ and is such that $\lambda(x)=\ell$ for all $x\in X$, then there exists 
$Y\subseteq X$ such that $\lambda(\sum_{y \in Y} y) \geq \ell + k$.
\end{lemma}

\begin{proof}
The Lemma is established by a straightforward induction on $k$. We give the details for completeness.

For the base case, let $k=0$ and let $X=\{x\}$ be a finite set of cardinality $2^0$ such that $\lambda(x) = \ell$.
Obviously choosing $Y=X$ gives the desired solution.


For the inductive step, let $k\geq 0$ and let $X$ be a set of cardinality $2^{k+1}$ such that for all $x \in X$ we have $\lambda(x)= \ell$. Let $A$ and $B$
be two disjoint subsets of $X$ each of cardinality $2^k$. By inductive hypothesis there exists $A' \subseteq A$ such that $\lambda(\sum_{a \in A'} a)\geq \ell + k$ and there exists $B'\subseteq B$
such that $\lambda(\sum_{b\in B'} b) \geq \ell + k$. We distinguish the following cases. If $\lambda(\sum_{a \in A'} a) = \ell + k$ and $\lambda(\sum_{b\in B'} b) = \ell + k$ then 
$\lambda (\sum_{ c \in A' \cup B'} c) \geq \ell + k +1$. If either $\lambda(\sum_{a \in A'} a) > \ell + k$ or $\lambda(\sum_{b\in B'} b) > \ell + k$ then we are done.

The argument can be carried out in $\RCA_0$
since quantification over finite sets formally means quantification over their numerical codes and the set $Y$ is a finite
subset of the finite set $X$, so that the existential quantifier over $Y$ is bounded. The induction predicate is 
then $\Pi^0_1$, and $\Pi^0_1$-induction holds in $\RCA_0$.
\end{proof}

The following Lemma appears as Lemma 9.9.6 in Dzhafarov and Mummert~\cite{Dza-Mum:22}. 
As pointed out by one of the reviewers of the present paper, there is an error in the 
proof in \cite{Dza-Mum:22} (where it is assumed that the element denoted by $x_2$ is in $\FS(I)$). 
We give an alternative argument, using Lemma \ref{lem:counting}.

\begin{lemma} \label{lem:apart}
The following is provable in $\RCA_0 + \RT^1$: For every $m \in \Nat$ and every infinite $I \subseteq \Nat$, 
there exists $x \in \FS(I)$ with $\lambda(x) \geq m$. 
\end{lemma}

\begin{proof}
 Fix $m$ and $I$ and suppose that every $x \in \FS(I)$ satisfies $\lambda(x) < m$. In particular this implies that every $x \in I$ satisfies $\lambda(x) < m$, since $I \subseteq \FS(I)$. By $\RT^1$ there exists an $\ell < m$ 
and an infinite set $J \subseteq I$ such that
$\lambda(x) = \ell$ for all $x \in J$.

Since $\ell < m$ there exists $k$ such that $\ell + k = m$. Pick a subset $X\subseteq J$ of cardinality $2^k$. Then by Lemma \ref{lem:counting} there exists a $Y\subseteq X$ such that
$\lambda(\sum_{y \in Y} y) \geq \ell + k = m$. This contradicts the hypothesis that $\lambda(x) < m$ for all $x \in \FS(I)$ .

\end{proof}

As a corollary one obtains the following Proposition, which will be used to show that $\HT$ self-strengthens to $\HT[\ap]$ over
$\RCA_0$.

\begin{proposition}\label{prop:self}
\hfill
\begin{enumerate}
\item The following is provable in $\RCA_0 + \RT^1$: For every infinite set $I\subseteq\Nat$, there is an infinite set $J$ such that $J$ is apart and $\FS(J) \subseteq \FS(I)$. 
\item For all infinite set $I\subseteq$ of natural numbers there exists an infinite set $J$ of natural numbers computable in $I$ such that $J$ is apart and $\FS(J)\subseteq \FS(I)$. 
\end{enumerate}
\end{proposition}

\begin{proof}
Define a sequence of elements $x_0 < x_1 < \cdots$ in $\FS(I)$ recursively as follows. Let $x_0 = \min(I)$. Given $x_i$ for some $i\in\Nat$, 
let $x_{i+1}$ be the least element of $\FS(I \setminus [0, x_i])$ such that $\lambda(x_{i+1}) > \mu(x_i)$.
The existence of $x_{i+1}$ follows from Lemma \ref{lem:apart}. Let $J = \{ x_i \,:\, i \in\Nat\}$. By construction $J$ is apart and $\FS(J)\subseteq \FS(I)$.
\end{proof}

Proposition \ref{prop:self} is close in both statement and proof to Corollary 9.9.8 in \cite{Dza-Mum:22} but ensures $\FS(J)\subseteq \FS(I)$
rather than $J \subseteq \FS(I)$ as in \cite{Dza-Mum:22}. This stronger condition is indeed needed in the proof of the following 
corollary, which appears as Theorem 9.9.9 in \cite{Dza-Mum:22}. The proof of the latter contains an error when it is claimed that
$J \subseteq \FS(I)$ implies $\FS(J) \subseteq \FS(I)$. 

\begin{corollary}\label{cor:HTap}
$\HT$ implies $\HT[\ap]$ over $\RCA_0$. Moreover $\HT \geq_{\sW}\HT[\ap]$.
\end{corollary}

\begin{proof}
From Proposition \ref{prop:self} and the fact that $\HT$ trivially implies $\RT^1$ over $\RCA_0$. 
Let $c:\Nat \to k$. Let $I$ be a solution to $\HT$ for $c$. By Proposition \ref{prop:self} there exists an infinite $J$ such that
$\FS(J)\subseteq \FS(I)$ and $J$ is apart. 

It is clear from the proof of Proposition \ref{prop:self} that there is a Turing functional that computes $J$ from $I$ uniformly.
This is sufficient to establish the claimed strong Weihrauch reduction.
\end{proof}

It is natural to ask whether Taylor's Theorem satisfies a similar self-strengthening with respect to the apartness 
condition. A positive answer is expected by considering the finite unions version of the theorem. Yet to establish the result in 
$\RCA_0$ the situation has to be analyzed more closely as we have done above for Hindman's Theorem. 
As observed by one of the reviewers of the present paper, the above argument does not immediately apply to 
the case of Taylor's Theorem. Indeed, what the min-term (or max-term) of a number is depends on whether that number
is seen as a sum of elements of $I$ or as a sum of elements of $J$, in the notation of Proposition \ref{prop:self} above.
Nevertheless Taylor's Theorem {\em does} imply its own self-strenghtening with apartness, as we next prove. 

\begin{theorem}\label{thm:canap}
$\canHT$ implies $\canHT[\ap]$ over $\RCA_0$.  Moreover, $\canHT\geq_\sW \canHT[\ap]$.
\end{theorem}

\begin{proof}
Given $c: \Nat \to \Nat$, let $H = \{h_0 < h_1 < \cdots \}$ be a solution to $\canHT$ for $c$.
Let $H' = \{ h'_1 < h'_2 < \cdots \}$ be an infinite apart set such that $\FS(H') \subseteq \FS(H)$ (defined as the set $J$ in the proof of Proposition \ref{prop:self}.

For each $i \in \Nat$, let $A_i \in \FIN(\Nat)$ be such that $\sum_{a \in A_i} h_a = h'_i$  and $h_{\min(A_i)} > h'_{i-1}$ if $i>0$. A non-empty set with these properties exists by definition of $H'$. We fix a uniform computable method to select $A_i$ if more than one choice exists (for instance, we take the set $A$ that satisfies the conditions above and that minimizes $\sum_{a \in A} 2^a$). Then, we can state the following three Claims.

\begin{claim}\label{cl:seq-propr}
For any set of indexes $I = \{i_0 < i_1 < \cdots < i_m\} \in \FIN(\Nat)$, the following properties hold:
\begin{enumerate}[(i)]
\item\label{propr-block} $A_{i_0} < A_{i_1} < \cdots < A_{i_m}$.
\item\label{propr-min} $\min(\bigcup_{i \in I} A_i) = \min(A_{i_0})$.
\item\label{propr-max} $\max(\bigcup_{i \in I} A_i) = \max(A_{i_m})$.
\item\label{propr-sum} $\sum_{i \in I} h'_i = \sum_{s \in \bigcup_{i \in I} A_i} h_s$.
\end{enumerate}

\end{claim}

\begin{proof}
\ref{propr-block} derives from the fact that, for any $s \in (0,m]$, $h_{\min(A_{i_s})} > h'_{i_s-1} \geq h'_{i_{s-1}} \geq h_{\max(A_{i_{s-1}})}$, which implies $\min(A_{i_s}) > \max(A_{i_{s-1}})$ because $H$ is enumerated in increasing order.

\ref{propr-min}, \ref{propr-max}, and \ref{propr-sum} are trivial consequences of \ref{propr-block}.
\end{proof}

\smallskip

\begin{claim}\label{propr:min-cons}
For any $I = \{i_0 < i_1 < \cdots < i_m\} \in \FIN(\Nat)$ and $J=\{j_0 < j_1 < \cdots < j_n\} \in \FIN(\Nat)$, $\min(I)=\min(J)$ if and only if $\min(\bigcup_{i \in I} A_i)=\min(\bigcup_{j \in J} A_j)$.
\end{claim}

\begin{proof}
($\Longrightarrow$) By hypothesis, $i_0 = j_0$, hence $A_{i_0} = A_{j_0}$ and $\min(A_{i_0}) = \min(A_{j_0})$. Then, by Claim \ref{cl:seq-propr}.\ref{propr-min}, $\min(\bigcup_{i \in I} A_i)=\min(\bigcup_{j \in J} A_j)$.

\smallskip

($\Longleftarrow$) By hypothesis, $\min(\bigcup_{i \in I} A_i)=\min(\bigcup_{j \in J} A_j)$ so, by Claim \ref{cl:seq-propr}.\ref{propr-min}, $\min(A_{i_0}) = \min(A_{j_0})$ and then $h_{\min(A_{i_0})} = h_{\min(A_{j_0})}$. Thus, we can show that $i_0 = j_0$, i.e., $\min(I) = \min(J)$. Assume otherwise, and suppose $i_0 < j_0$ (the case $i_0 > j_0$ is analogous). By definition of $A_{j_0}$, we can derive $h_{\min(A_{j_0})} > h'_{j_0-1} \geq h'_{i_0} \geq h_{\min(A_{i_0})}$, hence contradicting $h_{\min(A_{i_0})} = h_{\min(A_{j_0})}$.
\end{proof}

\smallskip

\begin{claim}\label{propr:max-cons}
For any $I = \{i_0 < i_1 < \cdots < i_m\} \in \FIN(\Nat)$ and $J=\{j_0 < j_1 < \cdots < j_n\} \in \FIN(\Nat)$, $\max(I)=\max(J)$ if and only if $\max(\bigcup_{i \in I} A_i) = \max(\bigcup_{j \in J} A_j)$.
\end{claim}

\begin{proof}
($\Longrightarrow$) By hypothesis, $i_m = j_n$, hence $A_{i_m} = A_{j_n}$ and $\max(A_{i_m}) = \max(A_{j_n})$. Then, by Claim~\ref{cl:seq-propr}.\ref{propr-max}, $\max(\bigcup_{i \in I} A_i) = \max(\bigcup_{j \in J} A_j)$.

\smallskip

($\Longleftarrow$) By hypothesis, $\max(\bigcup_{i \in I} A_i) = \max(\bigcup_{j \in J} A_j)$ so, by Claim~\ref{cl:seq-propr}.\ref{propr-max}, $\max(A_{i_m}) = \max(A_{j_n})$ and then $h_{\max(A_{i_m})} = h_{\max(A_{j_n})}$. Thus, we can show that $i_m = j_n$, i.e., $\max(I) = \max(J)$. Assume otherwise, and suppose $i_m < j_n$ (the case $i_m > j_n$ is analogous). By definition of $A_{j_n}$, we can derive $h_{\max(A_{j_n})} \geq h_{\min(A_{j_n})} > h'_{j_n-1} \geq h'_{i_m} \geq h_{\max(A_{i_m})}$, hence contradicting $h_{\max(A_{i_m})} = h_{\max(A_{j_n})}$.
\end{proof}

\medskip

Now we can show that $H'$ is a solution to $\canHT$ for $c$ by analyzing each case of Definition~\ref{thm:taylor}.

\medskip

\emph{Case 1.}  For any  $I, J \in \FIN(\Nat)$, by homogeneity of $H$ and by Claim~\ref{cl:seq-propr}.\ref{propr-sum}, $c(\sum_{i \in I} h'_i) = c(\sum_{s \in \bigcup_{i \in I} A_i} h_s) = c(\sum_{t \in \bigcup_{j \in J} A_j} h_t) = c(\sum_{j \in J} h'_j)$.

\medskip

\emph{Case 2.} Let $I, J \in \FIN(\Nat)$. If $I=J$, then $c(\sum_{i \in I} h'_i) = c(\sum_{j \in J} h'_j)$. Now assume $I \neq J$, as witnessed by $w \in I \setminus J$ (the case $w \in J \setminus I$ is analogous). By Claim \ref{cl:seq-propr}.\ref{propr-block} applied to $J \cup \{w\}$, we have that $A_w \cap A_j = \emptyset$ for all $j \in J$, therefore $\bigcup_{i \in I} A_i \neq \bigcup_{j \in J} A_j$.

Then, $c(\sum_{i \in I} h'_i) = c(\sum_{s \in \bigcup_{i \in I} A_i} h_s) \neq c(\sum_{t \in \bigcup_{j \in J} A_j} h_t) = c(\sum_{j \in J} h'_j)$, where the two equalities hold by  Claim~\ref{cl:seq-propr}.\ref{propr-sum}, while the inequality holds by Case 2 of Definition~\ref{thm:taylor}, since $c$ is applied to sums of different elements in $H$ on the two sides of the equality, as we noted above.

\medskip

\emph{Case 3.} Let $I, J \in \FIN(\Nat)$. If $\min(I)=\min(J)$, then $c(\sum_{i \in I} h'_i) = c(\sum_{s \in \bigcup_{i \in I} A_i} h_s) = c(\sum_{t \in \bigcup_{j \in J} A_j} h_t) = c(\sum_{j \in J} h'_j)$, where the first and the last equality hold by  Claim~\ref{cl:seq-propr}.\ref{propr-sum}, while the second equality holds by Case 3 of Definition~\ref{thm:taylor}, since in both sides of the equality, $c$ is applied to sums of elements in $H$ having the same minimum term by Claim~\ref{propr:min-cons}. Similarly, if $\min(I') \neq \min(J')$, we have $c(\sum_{i \in I} h'_i) = c(\sum_{s \in \bigcup_{i \in I} A_i} h_s) \neq c(\sum_{t \in \bigcup_{j \in J} A_j} h_t) = c(\sum_{j \in J} h'_j)$.

\medskip
\emph{Case 4.}  The proof is similar to the proof of Case 3, but using Claim~\ref{propr:max-cons} in place of Claim~\ref{propr:min-cons}.

\medskip

\emph{Case 5.} The proof is analogous to the proof of Cases 3 and 4.

\medskip

\end{proof}

As observed in~\cite{Kan-McA:87}, when the Canonical Ramsey's Theorem is applied to regressive functions 
the Regressive Ramsey's Theorem is obtained. 
Similarly, a regressive version of Hindman's Theorem follows from Taylor's Theorem.
We introduce the suitable versions of the notions of regressive function and min-homogeneous set. 

\begin{definition}[$\lambda$-regressive function]
A function $c:\Nat\to\Nat$ is called $\lambda$-regressive if and only if, for all $n \in\Nat$, $c(n) < \lambda(n)$
if $\lambda(n)>0$ and $c(n)=0$ if $\lambda(n)=0$.
\end{definition}

Obviously every $\lambda$-regressive function is regressive since $\lambda(n) < n$
for $n\in\Nat^+$. 

\begin{definition}[Min-term-homogeneity for $\FS$]
Let $c:\Nat\to \Nat$ and $H=\{h_0 < h_1 < \cdots\}\subseteq \Nat$. 
We call $\FS(H)$ min-term-homogeneous for $c$
if and only if, 
for all $I, J \in \FIN(\Nat)$, if $\min(I)=\min(J)$ then 
$c(\sum_{i \in I} h_i) = c(\sum_{j\in J} h_j)$.
\end{definition}

The following is an analogue of Kanamori-McAloon's Regressive Ramsey's Theorem in the spirit of Hindman's Theorem.

\begin{definition}[Regressive Hindman's Theorem]
We denote by $\lambda\regHT$ the following principle. For all $\lambda$-regressive $c:\Nat \to \Nat$ 
there exists an infinite
$H\subseteq\Nat$ such that $\FS(H)$ is min-term-homogeneous.
\end{definition}

For technical convenience we will always assume that $H \subseteq \Nat^+$.
In this paper we do not investigate optimal upper bounds on $\canHT$ and $\lambda\regHT$.

We start by observing how Taylor's Theorem implies the Regressive Hindman's Theorem just as the Canonical Ramsey's Theorem implies the Kanamori-McAloon Regressive Ramsey's Theorem. 

\begin{proposition}\label{prop:canHTtoregHTap}
$\canHT$ implies $\lambda\regHT$ over $\RCA_0$. Moreover, $\canHT\geq_\mathrm{\sW} \lambda\regHT$.
\end{proposition}

\begin{proof}
Let $c:\Nat\to \Nat$ be a $\lambda$-regressive function. By $\canHT$ there exists an infinite set $H\subseteq \Nat^+$ such that one of the five canonical cases occurs for $\FS(H)$.
It is easy to see that, since $c$ is $\lambda$-regressive, only case (1) and case (3) of Definition \ref{thm:taylor} can occur. Thus $\FS(H)$ is min-term-homogeneous for $c$.
\end{proof}

Similarly to Hindman's Theorem and Taylor's Theorem, the Regressive Hindman's Theorem self-improves to its own version with apartness, as shown below. 
We first show that $\lambda\regHT$ implies the Infinite Pigeonhole Principle.

\begin{lemma} \label{regHTtoRT1}
$\lambda\regHT$ implies $\RT^1$  over $\RCA_0$.
\end{lemma}

\begin{proof}
Given $f: \Nat \to k$, with $k\geq 1$, let $g: \Nat \to \Nat$ be defined as follows:

\begin{equation*}
g(n) =
\begin{cases}
\lambdam(n) & \text{if } \lambdam(n) < k,\\
f(n) & \text{otherwise,}
\end{cases}
\end{equation*}
where $\lambdam(n) = \lambda(n)-1$ if $\lambda(n)>0$, otherwise $\lambdam(n) = 0$.

Clearly, $g$ is $f$-computable and $\lambda$-regressive, so let $H = \{h_0 < h_1 < \cdots\}$ be a solution to $\lambda\regHT$ for $g$.
First, we prove the following Claim.


\begin{claim*}
There exists an infinite $H' = \{h'_0 < h'_1 < \cdots\} \subseteq H$ such that $\lambdam(h'_{n_1} + h'_{n_2} + h'_{n_3} + h'_{n_4}) \geq k$ for all $n_1 < n_2 < n_3 < n_4$.
\end{claim*}

\begin{proof} \let\qed\relax
Let us define $J = \{j \in H\ |\ \lambdam(j) < k\}$. If $J$ contains finitely many elements, then $(H \setminus J)$ witnesses the existence of $H'$. 
Thus, let us assume $J = \{j_0 < j_1 < \cdots\}$ is infinite.

Notice that the sequence $\lambdam(j_0), \lambdam(j_1), \dots$ never decreases: suppose otherwise by way of contradiction, and let $j,j' \in J$ be such that $j<j'$ and $\lambdam(j) > \lambdam(j')$. Then 
we have $g(j) = \lambdam(j) > \lambdam(j') = \lambdam(j+j') = g(j+j')$;  this contradicts the min-term-homogeneity of $\FS(H)$. Hence $\lambdam$ on $J$ is a bounded non-decreasing function on an infinite set. 

Then we have two cases. Either for any $j \in J$ there exists $j'>j$ in $J$ such that $\lambdam(j')>\lambdam(j)$, or there exists $j \in J$ such that, for any $j'>j$ in $J$, $\lambdam(j) \geq \lambdam(j')$.
The former case can not hold, since by definition of $J$, $\lambdam(j)<k$ for any $j \in J$.

In the latter case, instead, we have some $m \in J$ such that $\lambdam(m) \geq \lambdam(j)$ for any $j$ in $J$. Since $\lambdam(j_0), \lambdam(j_1), \ldots$ is non-decreasing, $\lambdam(j) = \lambdam(m)$ holds for each $j$ in the infinite set $J' = J \setminus [0,m)$. Finally, we can show that $J'$ witnesses the existence of $H'$. Assume otherwise by way of contradiction. Then, there exist $j,j',j'',j''' \in J'$ such that $j<j'<j''<j'''$ and $\lambda'(j+j'+j''+j''')<k$. Thus $g(j+j'+j''+j''') = \lambdam(j+j'+j''+j''')$ by definition of $g$. On the other hand, since $j\in J'\subseteq J$, $\lambdam(j)<k$ and therefore $g(j) = \lambdam(j)$ by definition of $g$. Moreover, $\lambdam(j)=\lambdam(j')=\lambdam(j'')=\lambdam(j''')$ since $j,j',j'',j'''\in J'$. Therefore we have the following inequality 
$$g(j+j'+j''+j''') = \lambdam(j+j'+j''+j''') > \lambdam(j) = g(j),$$
contradicting the min-term-homogeneity of $\FS(H)$. This completes the proof of the Claim. Notice that, while  $\lambda(x)=\lambda(y)$ implies $\lambda(x+y)>\lambda(x)$ for any $x,y \in \Nat^+$, the same implication does not hold when using $\lambdam$: hence, sums of 4 elements are required in the argument above.
\end{proof}

In order to prove the lemma, let $H' = \{h'_0 < h'_1 < \cdots\}$ be as in the previous Claim. Then, for any $n_0 < n_1 < n_2$ in $\Nat^+$, we have
\begin{equation*}
\begin{split}
f(h'_0 + h'_{n_0} + h'_{n_1} + h'_{n_2}) & = g(h'_0 + h'_{n_0} + h'_{n_1} + h'_{n_2})\\
 & = g(h'_0 + h'_1 + h'_2 + h'_3) \\ 
 & = f(h'_0 + h'_1 + h'_2 + h'_3),
\end{split}
\end{equation*}

where the first and the last equalities hold by the previous Claim and by definition of $g$, while the second equality holds by min-term-homogeneity of $\FS(H)$.

Hence $\{(h'_0 + h'_{n_0} + h'_{n_1} + h'_{n_2}) \ |\ 0 < n_1 < n_2 < n_3\}$ is an infinite homogeneous set for $f$.
\end{proof}

\begin{proposition}\label{prop:regap}
$\lambda\regHT$ implies $\lambda\regHT[\ap]$ over $\RCA_0$. Moreover, $\lambda\regHT\geq_\sW \lambda\regHT[\ap]$.
\end{proposition}

\begin{proof}
The proof of Theorem~\ref{thm:canap} adapts {\em verbatim} to the case of $\lambda\regHT$. Lemma~\ref{regHTtoRT1} 
takes care of the use of $\RT^1$ in that proof, which is only needed for the implication over $\RCA_0$. 
\end{proof}

It is easy to see that the proof of Lemma \ref{regHTtoRT1} uses only sums of at most 4 terms.
However, this does not help in extending the previous Proposition to some restriction of $\lambda\regHT$ (see section \ref{sec:restrictions}), since the proof of Theorem \ref{thm:canap} still requires sums of arbitrary length.

\smallskip

The following proposition shows that the Regressive Hindman's Theorem implies Hindman's Theorem.

\begin{proposition}\label{prop:reghtht}
$\lambda\regHT$ implies $\HT$ over $\RCA_0$.
\end{proposition}

\begin{proof}

Given $f: \Nat \to k$, with $k\geq 1$, and let $g:\Nat \to k$ be as follows: 
$$
    g(n) =
    \begin{cases}
      f(n) & \mathrm{if} \; f(n) < \lambda(n), \\
      0        & \text{otherwise.}
    \end{cases}
$$
The function $g$ is $\lambda$-regressive by construction and obviously $f$-computable. Let $H=\{h_0 < h_1 <  \cdots\}\subseteq\Nat^+$ 
be an infinite set such that $FS(H)$ is
min-term-homogeneous for $g$. By Proposition~\ref{prop:regap} we can assume that $H$ is apart. 
Let $i$ be the minimum such that $\lambda(h_i) > k$. 
Let $H^{-} = H \setminus \{h_0, \dots, h_i\}$. By choice of $H^-$, $g$ behaves like $f$ on 
$\FS(H^-)$. Let $g^-$ be the $k$-colouring of numbers induced by $g$ on $H^-$.

By $\RT^1_k$ (which we can assume by Lemma~\ref{regHTtoRT1}) let $H'= \{h'_0 <  h'_1 < \cdots\}$ be an infinite subset of $H^-$ homogeneous for $g^-$. Then, 
for $\{s_1, \dots, s_m\}$ and $\{t_1, \dots, t_n\}$ non-empty subsets of $H'$, we have
\begin{equation*}
\begin{split}
f(s_1+\cdots+s_m) & = g(s_1 + \cdots + s_m)\\
 & = g(s_1)  = g^-(s_1)\\
 & = g^-(t_1)  = g(t_1) \\ 
 & = g(t_1+\cdots + t_n) \\
 & = f(t_1+\cdots + t_n),
\end{split}
\end{equation*}
since $\FS(H^-)$ is min-term-homogeneous for $g$ and $g$ coincides with $f$ on $\FS(H^-)$. 
\end{proof}

We do not know if the implication in Proposition~\ref{prop:reghtht} can be reversed. In the next section we will observe that $\RT^1_k$ can be Weihrauch-reduced to some restriction of $\lambda\regHT$ with apartness -- hence, a fortiori, it can be Weihrauch-reduced to $\lambda\regHT$ (see Proposition~\ref{prop:regHTtoHTeqn} {\em infra}).

\section{Restrictions of the Regressive Hindman's Theorem} \label{sec:restrictions}

Restrictions of Hindman's Theorem relaxing the monochromaticity requirement to particular families of 
finite sums received substantial attention in recent years (see~\cite{Car:2021} for an overview and bibliography).
Two natural families of restrictions of Hindman's Theorem are obtained by restricting the number of terms in the monochromatic sums. We introduce the needed terminology. For $X\subseteq \Nat$ and $n\in\Nat^+$ we denote by $\FS^{\leq n}(X)$ the set of all non-empty sums of at most $n$ distinct elements of $X$; we denote by $\FS^{=n}(X)$ the set of all sums of exactly $n$ distinct elements of $X$. 

\begin{definition}[Bounded Hindman's Theorems]
Let $n, k\in\Nat^+$. We denote by $\HT^{\leq n}_k$ (resp. $\HT^{=n}_k$) the following principle.
For every $c:\Nat \to k$ there exists an infinite set $H\subseteq\Nat$ such that $\FS^{\leq n}(H)$ (resp. $\FS^{=n}(H)$)
is monochromatic for $c$.\\
We use $\HT^{\leq n}$ (resp. $\HT^{=n}$) to denote $(\forall k \geq 1)\,\HT^{\leq n}_k$ (resp. $(\forall k \geq 1)\,\HT^{=n}_k$).
\end{definition}

Note that $\HT^{\leq 1}_k$, $\HT^{=1}_k$ and $\RT^1_k$ are all equivalent and strongly Weihrauch inter-reducible (by identity). 

To formulate analogous restrictions of $\lambda\regHT$ we extend the definition of min-term-homogeneity in the natural way. For $n\geq 1$, we denote by $\FIN^{\leq n}(\Nat)$ (resp. $\FIN^{=n}(\Nat)$) the set of all non-empty subsets of $\Nat$ of cardinality at most $n$ (resp. of cardinality $n$).

\begin{definition}[Min-term-homogeneity for $\FS^{\leq n},\FS^{=n}$]
Let $n\in\Nat^+$. Let $c:\Nat\to \Nat$ be a colouring and $H=\{h_0 < h_1 < \cdots \}$ an infinite subset
of $\Nat$. 
We call $\FS^{\leq n}(H)$ (resp.~$\FS^{=n}(H)$) min-term-homogeneous for $c$
if and only if, 
for all $I, J \in \FIN^{\leq n}(\Nat)$ (resp.~$I, J \in \FIN^{=n}(\Nat)$), if $\min(I)=\min(J)$ then 
$c(\sum_{i \in I} h_i) = c(\sum_{j\in J} h_j)$.
\end{definition}

We can then formulate the natural restrictions of the Regressive Hindman's Theorem obtained by relaxing the min-term-homogeneity requirement from $\FS(H)$ to $\FS^{\leq n}(H)$ or $\FS^{=n}(H)$. For example,  $\lambda\regHT^{\leq n}$ is defined as $\lambda\regHT$ with $\FS^{\leq n}(H)$ replacing $\FS(H)$. 

\begin{definition}[Bounded $\lambda$-Regressive Hindman's Theorems]
Let $n\in\Nat^+$. We denote by $\lambda\regHT^{\leq n}$ (resp.~$\lambda\regHT^{=n}$) the following principle. For all $\lambda$-regressive $c:\Nat \to \Nat$ there exists an infinite
$H\subseteq\Nat$ such that $\FS^{\leq n}(H)$ (resp.~$\FS^{=n}$) is min-term-homogeneous for $c$.
\end{definition}

Note that $\lambda\regHT^{\leq 1}$ and $\lambda\regHT^{=1}$ are trivial. 
We also point out the following obvious relations: $\lambda\regHT$ yields $\lambda\regHT^{\leq n}$ which yields
$\lambda\regHT^{=n}$ for all $n$ (both in $\RCA_0$ and by strong Weihrauch reductions) and similarly for the 
versions with the apartness condition. 
Also, for $m > n$, $\lambda\regHT^{\leq m}$ obviously yields $\lambda\regHT^{\leq n}$, while $\lambda\regHT^{=m}$ yields $\lambda\regHT^{=n}$ if $m$ is a multiple of $n$ (see the analogous results for Hindman's Theorem for sums of exactly $n$ terms in~\cite{Car-Kol-Lep-Zda:20}, Proposition 3.5).

\subsection{Bounded regressive Hindman's Theorems and Ramsey-type principles}

We compare the bounded versions of our regressive Hindman's Theorem with other prominent Ramsey-type and Hindman-type principles.

We start with the following simple Lemma showing that, for every $n\geq 2$, $\lambda\regHT^{=n}[\ap]$ implies $\RT^1$. Note that in Lemma~\ref{regHTtoRT1} we established that $\lambda\regHT$ {\em without apartness} implies $\RT^1$ and we later used this result to show that $\lambda\regHT$ implies $\lambda\regHT[\ap]$ (Proposition~\ref{prop:regap}).

\begin{lemma}\label{lem:rt1k}
Let $n\geq 2$. Over $\RCA_0$, $\lambda\regHT^{=n}[\ap]$ implies $\RT^1$. Moreover, for any $k\in\Nat^+$, we have $\RT^1_k \leq_\mathrm{\sW} \lambda\regHT^{=n}[\ap]$.
\end{lemma}

\begin{proof}
We give the proof for $n=2$ for ease of readability. Let $f:\Nat \to k$ be given, with $k \geq 1$. Define $g:\Nat \to k$ as follows.
$$
    g(m) =
    \begin{cases}
      0 & \mathrm{if} \; \lambda(m) \leq k, \\
      f(\mu(m))	&	\text{otherwise.}
    \end{cases}
$$
Clearly $g$ is $\lambda$-regressive and $f$-computable in a uniform way. 
Let $H=\{h_0 < h_1 < \cdots\}\subseteq\Nat^+$ be an infinite apart set of positive integers such that $\FS^{=2}(H)$ is min-term-homogeneous for $g$.

By the apartness condition, for all $h \in H\setminus \{h_0, h_1, \ldots,h_k\}$ we have $g(h) = f(\mu(h))$. Then it is easy to see that $M=\{\mu(h_{k+2}),\mu(h_{k+3}),\ldots\}$ is an infinite $f$-homogeneous set of colour $f(\mu(h_{k+2}))$ since, for any $i$, $f(\mu(h_{k+2+i})) = g(h_{k+1}+h_{k+2+i}) = g(h_{k+1}+h_{k+2}) = f(\mu(h_{k+2}))$.
\end{proof}


The following proposition relates the principles $\lambda\regHT^{= n}[\ap]$ (respectively $\lambda\regHT^{\leq n}[\ap]$) with the principles $\HT^{= n}_k[\ap]$ (respectively $\HT^{\leq n}_k[\ap]$). 
The proof is essentially the same as the proof of Proposition~\ref{prop:reghtht}. 

\begin{proposition}\label{prop:regHTtoHTeqn}
Let $n\geq 2$. 
\begin{enumerate}
\item $\lambda\regHT^{= n}[\ap]$ implies $\HT^{= n}[\ap]$ over $\RCA_0$. Moreover, for any $k\in\Nat^+$, $\lambda\regHT^{= n}[\ap] \geq_\mathrm{c}  \HT^{= n}_k[\ap]$.
\item $\lambda\regHT^{\leq n}[\ap]$ implies $\HT^{\leq n}[\ap]$ over $\RCA_0$. Moreover, for any $k\in\Nat^+$, $\lambda\regHT^{\leq n}[\ap] \geq_\compred  \HT^{\leq n}_k[\ap]$.
\end{enumerate}
\end{proposition}

\begin{proof}
We prove the second point, the proof of the first point being completely analogous. Given $f: \Nat \to k$, with $k\in\Nat^+$, let $g:\Nat \to k$ be as follows: 
$$
    g(m) =
    \begin{cases}
      f(m) & \mathrm{if} \; f(m) < \lambda(m), \\
      0        & \text{otherwise.}
    \end{cases}
$$
The function $g$ is $\lambda$-regressive and $f$-computable. 
By $\lambda\regHT^{\leq n}[\ap]$ let $H=\{h_0 < h_1 <  \cdots\}\subseteq\Nat^+$ be an infinite apart set such that $\FS^{\leq n}(H)$ is
min-term-homogeneous for $g$. Let $g': H \setminus \{h_0, \dots, h_{k-1}\} \to k$ be defined as 
$g'(h_i)=g(h_i + h_{i+1} + \cdots + h_{i+n-1})$. 

By $\RT^1_k$, let $H' \subseteq H$ be an infinite homogeneous set for $g'$. For the sake of establishing the implication over $\RCA_0$, 
recall that $\RT^1$ follows from $\lambda\regHT^{=2}[\ap]$ by Lemma \ref{lem:rt1k} and therefore also from $\lambda\regHT^{\leq n}[\ap]$ for any $n\geq 2$. For the sake of the computable reduction result, just notice that for each 
fixed $k\in \Nat^+$, $\RT^1_k$ is computably true. Then, 
for $\{s_1, \dots, s_p\}$ and $\{t_1, \dots, t_q\}$ non-empty subsets of $H'$, with $p,q \leq n$ and $s_1 < \dots < s_p$, $t_1 < \dots < t_q$, we have
\begin{equation*}
\begin{split}
f(s_1+\cdots+s_p) & = g(s_1+\cdots+s_p)\\  
& \stackrel{(*)}{=} g(s_1) = g'(s_1) \\
& = g'(t_1) = g(t_1) \\
& \stackrel{(**)}{=} g(t_1+\cdots+t_q)\\ 
& = f(t_1+\dots+t_q),
\end{split}
\end{equation*}

where the equalities dubbed by $(*)$ and $(**)$ hold by the min-term-homogeneity of $\FS^{\leq n}(H)$ for $g$.
This shows that $H'$ is an apart solution to $\HT^{\leq n}_k$ for $f$.
\end{proof}

\begin{remark}\label{rem:reductions}
The previous proof gives us a hint as how to extend the reduction to $\HT^{\leq n}[\ap]$, i.e. to the universally-quantified principles $(\forall k \geq 1) \ \HT^{\leq n}_k[\ap]$. In that case, the number of colours is not given as part of the instance, and it cannot be computably inferred from the instance $X$ of the principle $\HT^{\leq n}[\ap]$ (see the discussion in~\cite{Dza-Mum:22} p.~54 for more details on this issue). Nevertheless, we can easily obtain a computable reduction by just observing that the proof of Proposition~\ref{prop:regHTtoHTeqn} provides us, for any $k \geq 1$, with both an $X$-computable procedure giving us an instance $\widehat{X}$ of $\lambda\regHT^{\leq n}[\ap]$, and an $(X \oplus \widehat{Y})$-computable procedure transforming a solution $\widehat{Y}$ for $\widehat{X}$ to a solution for $X$: so, even if we do not know the actual value of $k$, we know that the two procedures witnessing the computable reduction do exist. Thus, we can conclude that 
for any $n \geq 2$, $\lambda\regHT^{\leq n}[\ap] \geq_\compred  \HT^{\leq n}[\ap]$. It is not straightforward to improve this result to a Weihrauch reduction.

The same argument also applies to the case of $\lambda\regHT^{=n}[\ap]$, so that we have that 
for any $n \geq 2$, $\lambda\regHT^{=n}[\ap] \geq_\compred  \HT^{=n}[\ap]$.
\end{remark}

Also, we point out that a proof of $\lambda\regHT^{\leq 2}$ that does not also prove $\HT$ (or, more technically, a separation over $\RCA_0$ of these two principles) would answer Question 12 from \cite{Hin-Lea-Str:03}.

It is worth noticing that a further slight adaptation of the proof of Proposition \ref{prop:regHTtoHTeqn} gives a direct proof
of $\RT^n$ from $\REG^n$ and also shows that $\REG^{n} \geq_\compred \RT^{n}_k$. 
The following definition can be used for computably reducing $\RT^n_k$ to $\REG^n$ (for $n\geq 2$ and $k\in\Nat^+$). 
Given $c: [\Nat]^n \to k$, with $k\in\Nat^+$, let $c^+: [\Nat]^n \to k$ be as follows:

\begin{equation*}
c^+(x_1,\ldots,x_n) =
\begin{cases}
0 & \text{if } x_1 \leq k,\\
c(x_1,\ldots,x_n) & \text{otherwise.}
\end{cases}
\end{equation*}

We can thus state the following Proposition. 

\begin{proposition}\label{prop:regrt}
For any $n\geq 2$ and $k\in\Nat^+$, $\RT^n_k \leq_{\compred} \REG^n$.
\end{proposition}

Note that by $\HT^{=n}_k[\ap] \leq_\sW \RT^n_k$ (see~\cite{Car-Kol-Lep-Zda:20}), the above also implies $\HT^{=n}_k[\ap] \leq_{\compred} \REG^n$ for any $n \geq 2$ and $k \in \Nat^+$. 

\paragraph{Equivalents of $\ACA_0$.}

Proposition~\ref{prop:regHTtoHTeqn}, coupled with the fact that $\HT^{=3}_2[\ap]$ implies $\ACA_0$ (Theorem 3.3 in~\cite{Car-Kol-Lep-Zda:20}), yields the following corollary.

\begin{corollary}
$\lambda\regHT^{=3}[\ap]$ implies $\ACA_0$ over $\RCA_0$.
\end{corollary}

\begin{proof}
From Theorem 3.3 in~\cite{Car-Kol-Lep-Zda:20} and Proposition~\ref{prop:regHTtoHTeqn} above.
\end{proof}

We have the following reversal, showing that $\lambda\regHT^{= 3}[\ap]$ is a ``weak yet strong'' restriction of Taylor's Theorem in the sense of~\cite{Car:16:wys}.
The result is analogous to the implication from $\RT^n_k$ to $\HT^{=n}_k$ (see~\cite{Car-Kol-Lep-Zda:20}).

\begin{theorem}\label{thm:acareght}
Let $n\in\Nat^+$. $\ACA_0$ proves $\lambda\regHT^{= n}[\ap]$.
Moreover, $\lambda\regHT^{= n}[\ap] \leq_{\mathrm{\sW}} \REG^n$.
\end{theorem}

\begin{proof}
We give the proof for $n=2$ for ease of readability. 


Let $f:\Nat \to \Nat$ be $\lambda$-regressive. 
Let $g:[\Nat]^2 \to \Nat$ be defined as follows: $g(x,y) = f(2^x+2^y)$. 
The function $g$ is regressive since $f$ is 
$\lambda$-regressive. Recall that $\REG^2$ is provable in $\ACA_0$.
Let $H\subseteq \Nat^+$ be a min-homogeneous solution to $\REG^2$ for $g$. Let $\hat{H} = \{ 2^h \,:\, h \in H\}$. Obviously $\hat{H}$ is apart. It is easy to see that $\FS^{= 2}(\hat{H})$ is min-term-homogeneous for $f$: let $2^h < 2^{h'} < 2^{h''}$ be elements of $\hat{H}$. Then 
$$ f(2^h+2^{h'}) = g(h,h')=g(h,h'') = f(2^h+2^{h''}).$$

\end{proof}

We do not know if the reduction in Theorem~\ref{thm:acareght} can be reversed. 

We next show that $\lambda\regHT^{=2}[\ap]$ already implies Arithmetical Comprehension. The proof is reminiscent of the proof that $\HT^{\leq 2}_2[\ap]$ implies $\ACA_0$ in~\cite{Car-Kol-Lep-Zda:20}, but the use of $\lambda$-regressive colourings allows us to avoid the parity argument used in that proof. As happens in the proofs of independence of combinatorial principles from 
Peano Arithmetic \cite{Kan-McA:87}, in the present setting the use of regressive colourings simplifies the combinatorics.  

Let $\RAN$ be the $\Pi^1_2$ principle stating that for every injective function $f:\Nat\to \Nat$ the range of $f$ (denoted by $\rho(f)$) exists. It is well-known 
that $\RAN$ is equivalent to $\ACA_0$ (see~\cite{Sim:SOSOA}).

\begin{theorem}\label{thm:regtoaca}
Let $n\geq 2$.
$\lambda\regHT^{= n}[\ap]$ implies $\ACA_0$ over $\RCA_0$. Moreover, 
$$\lambda\regHT^{=n}[\ap]\geq_\W \RAN.$$
\end{theorem}

\begin{proof}
We give the proof for $n=2$. The easy adaptation to larger values is left to the reader. 

Let $f:\Nat\to\Nat$ be injective. For technical convenience and without loss of generality we assume that $f$ never takes the value $0$. We show, using $\lambda\regHT^{= 2}[\ap]$, that $\rho(f)$ (the range of $f$) exists. 

Define $c:\Nat\to \Nat$ as follows. If $m$ is a power of $2$ then $c(m)=0$. Else 
$c(m)=$ the unique $x$ such that $x < \lambda(m)$ 
and there exists $j \in [\lambda(m), \mu(m))$ such that $f(j)=x$ and for all $j < j' < \mu(m)$, 
$f(j') \geq \lambda(m)$. If no such $x$ exists, we set $c(m)=0$. 

Intuitively $c$ checks whether there are values below $\lambda(m)$ in the range of $f$ restricted
to $[\lambda(m), \mu(m))$. If any, it returns the latest one, i.e., the one obtained as image of the 
maximal $j \in [\lambda(m), \mu(m))$ that is mapped by $f$ below $\lambda(m)$). 
In other words, $x$ is the ``last'' element below $\lambda(m)$
in the range of $f$ restricted to $[\lambda(m), \mu(m))$. 

The function $c$ is computable in $f$ and $\lambda$-regressive.

Let $H = \{ h_0 < h_1 < \cdots \}\subseteq\Nat^+$ be an apart solution to $\lambda\regHT^{= 2}$ for $c$. Without loss of generality 
we can assume that $\lambda(h_0) > 1$, since $H$ is apart. Let $h_i \in H$. 

We claim that if $x < \lambda(h_i)$ and $x$ is in the range of $f$ then $x$ is in the range of $f$ restricted to $[0, \mu(h_{i+1}))$. 

We prove the claim as follows. Suppose, by way of contradiction, that there exist $h_i\in H$ and $x< \lambda(h_i)$ such that $x \in \rho(f)$ but $x \notin f([0, \mu(h_{i+1}))$. Let $b$ be the true bound for the elements in the range of $f$
smaller than $\lambda(h_i)$, i.e., $b$ is such that if $n < \lambda(h_i)$ and $n\in \rho(f)$, then $n < b$. The existence of $b$ follows in $\RCA_0$ from strong $\Sigma^0_1$-bounding (see 
\cite{Sim:SOSOA}, Exercise II.3.14):
$$ \forall n \exists b \forall i < n ( \exists j (f(j)=i) \to \exists j < b (f(j)=i)),$$
where we take $n=\lambda(h_i)$.

Let $h_j$ in $H$ be such that $h_j > h_{i+1}$ and $\mu(h_j) \geq b$. Such an $h_j$ exists since $H$ is infinite. 

Then, by min-term-homogeneity of $\FS^{=2}(H)$, $c(h_i + h_{i+1}) = c(h_i + h_j)$. But by choice of $h_i, x$ and $h_j$, 
and the definition of $c$, it must be the case that $c(h_i + h_{i+1}) \neq c(h_i + h_j)$. To see this, first note that, 
by apartness of $H$, the following equalities hold: 
$$\lambda(h_i + h_{i+1}) = \lambda(h_i) = \lambda(h_i+h_j),\ \mu(h_i + h_{i+1}) = \mu(h_{i+1}),\ \mu(h_i + h_j) = \mu(h_j).$$ 
Then observe that
$c(h_i + h_j) > 0$: by hypothesis $f^{-1}(x) \in [\mu(h_{i+1}), b)$ (recall that $f$ is injective), therefore $x$ is a value of $f$ below $\lambda(h_i+h_j)$ whose pre-image under $f$ is in $[\lambda(h_i+h_j),\mu(h_i+h_j))$, i.e.
in $[\lambda(h_i),\mu(h_j))$. Suppose now that $c(h_i + h_{i+1}) = z > 0$. Then, by definition of
$c$, it must be the case that $z < \lambda(h_i + h_{i+1})$, i.e., $z < \lambda(h_i)$, and $f^{-1}(z)$ is in 
$[\lambda(h_i+h_{i+1}), \mu(h_i+h_{i+1}))$, 
i.e. in $[\lambda(h_i), \mu(h_{i+1}))$. This $z$ cannot be the value of $c(h_i + h_j)$, since 
by hypothesis and by choice of $b$, we have $x < \lambda(h_i)$ and $f^{-1}(x)$ is in $[\mu(h_{i+1}), b)$, 
hence in $[\lambda(h_i+h_j), \mu(h_i+h_j))$. Thus $z$ cannot be the value of $f$ below $\lambda(h_i)$
with maximal pre-image under $f$ in $[\lambda(h_i+h_j),\mu(h_i+h_j))$ as the definition of $c(h_i+h_j)$ requires, 
since $f^{-1}(z) < \mu(h_{i+1}) \leq f^{-1}(x)$ and $f$ is injective. This concludes our reasoning by way of contradiction 
and hence establishes the claim that values in the range of $f$ below $\lambda(h_i)$ appear as values of $f$ applied to 
arguments smaller than $\mu(h_{i+1})$.

In view of the just established claim it is easy to see that the range of $f$ can be decided computably in $H$ as follows. Given $x$, pick any $h_i \in H$ such that 
$x < \lambda(h_i)$ and check whether $x$ appears in $f([0, \mu(h_{i+1}))$. 
\end{proof}

Theorem~\ref{thm:regtoaca} for the case of $n=2$ should be contrasted with the fact that $\HT^{=2}_2[\ap]$ follows easily from $\RT^2_2$ and is therefore strictly weaker than $\ACA_0$, while
$\HT^{=3}_2[\ap]$ implies $\ACA_0$ as proved in~\cite{Car-Kol-Lep-Zda:20}. 
The situation matches the one among $\REG^2$, $\RT^3_2$ and $\RT^2_2$ (see Theorem~\ref{thm:mileti}). 

The proof of Theorem~\ref{thm:regtoaca} can be recast in a straightforward way to show that there exists 
a computable $\lambda$-regressive colouring such that all apart solutions to $\lambda\regHT^{=2}$ for that colouring
compute the first Turing Jump $\emptyset'$. Analogously, the reduction can be cast in terms of the $\Pi^1_2$-principle
$\forall X \exists Y ( Y = (X)')$ expressing closure under the Turing Jump, rather than in terms of $\RAN$. 


The next theorem summarizes the implications over $\RCA_0$ for the Regressive Hindman's theorems
for sums of exactly $n$ elements, compared with other prominent Ramsey-theoretic principles (see  
Figure~\ref{fig:implicationsHT}).

\begin{theorem}
The following are equivalent over $\RCA_0$. 
\begin{enumerate}
\item $\ACA_0$.
\item $\REG^n$, for any fixed $n\geq 2$.
\item $\RT^n_k$, for any fixed $n\geq 3$, $k \geq 2$.
\item $\HT^{=n}_k[\ap]$, for any fixed $n \geq 3$, $k \geq 2$. 
\item $\lambda\regHT^{=n}[\ap]$, for any fixed $n \geq 2$.
\end{enumerate}
\end{theorem}

\begin{proof}
The equivalences between point (1), (2) and (3) are as in Theorem~\ref{thm:mileti}. The equivalence of 
(1) and (4) is from Proposition 3.4 in~\cite{Car-Kol-Lep-Zda:20}. Then the equivalence of (5) with points from (1) to (4) follows from Theorem~\ref{thm:acareght}, Theorem~\ref{thm:regtoaca} and Proposition~\ref{prop:regHTtoHTeqn}. 
\end{proof}

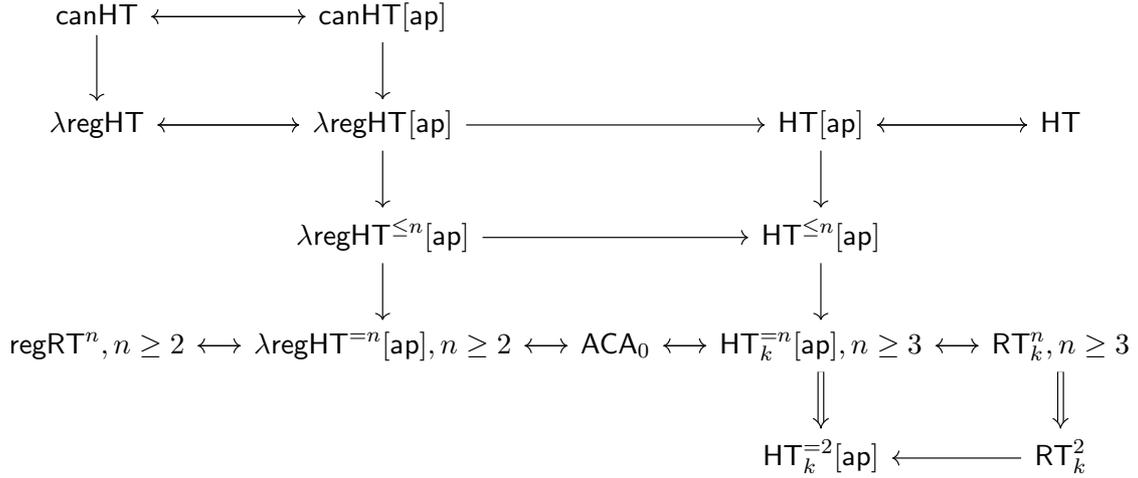
\begin{figure}
\centering
\[\begin{tikzcd}[row sep=2em,column sep=1.5em]
\arrow[d] \canHT \arrow[r] & \arrow[l] \canHT[\ap] \arrow[d] &   &
 &\\
 \lambda\regHT \arrow[r] & \lambda\regHT[\ap] \arrow[l]\arrow[d] \arrow[rr] & &  \HT[\ap] 
\arrow[d]\arrow[r] & \HT \arrow[l]\\
 &  \lambda\regHT^{\leq n}[\ap] \arrow[d] \arrow[rr] & & \HT^{\leq n}[\ap] \arrow[d] &  \\
\REG^n, n \geq 2 \arrow[r]& \arrow[l] \lambda\regHT^{= n}[\ap], n\geq 2 \arrow[r] & \ACA_0 \arrow[l] \arrow[r] & \HT^{=n}_k[\ap], n\geq 3 \arrow[l] \arrow[r] \arrow[d, Rightarrow] & \arrow[l] \RT^n_k, n\geq 3 \arrow[d, Rightarrow]\\
 & & & \HT^{=2}_k[\ap] & \arrow[l] \RT^2_k \end{tikzcd}\]
\caption{
Implications over $\RCA_0$. Double arrows indicate strict implications.
The equivalence of $\canHT[\ap]$ and $\canHT$ is from Theorem~\ref{thm:canap}. The implication from $\canHT$ to 
$\lambda\regHT$ is from Proposition~\ref{prop:canHTtoregHTap} and similarly for the versions with apartness. The equivalence between $\lambda\regHT$ and 
$\lambda\regHT[\ap]$ is from Proposition \ref{prop:regap}. The implication from $\lambda\regHT$ to $\HT$ is from Proposition \ref{prop:reghtht}. The implication from 
$\lambda\regHT^{\leq n}[\ap]$ to $\HT^{\leq n}[\ap]$ is from Proposition~\ref{prop:regHTtoHTeqn}. The equivalence of $\lambda\regHT^{=n}[\ap]$ with $\ACA_0$ (for $n\geq 2$) is from Theorem~\ref{thm:acareght} and Theorem \ref{thm:regtoaca}. The equivalence of $\HT^{=n}_k[\ap]$ with $\ACA_0$ (for $n\geq 3, k\geq 2$) is from \cite{Car-Kol-Lep-Zda:20}. The equivalence of $\RT^n_k$ with $\ACA_0$ (for $n\geq 3, k\geq 2$) is a classical result of Simpson, see Theorem III.7.6 in~\cite{Sim:SOSOA}.}\label{fig:implicationsHT}
\end{figure}

In terms of computable reductions we have the following, for $n\geq 2$ and $k\in\Nat^+$: 
$$ \RT^{2n-1}_2 \geq_\sW \REG^n \geq_\compred \RT^n_k,$$
where the first inequality is due to Hirst \cite{Hir:phd} and the second inequality is from Proposition~\ref{prop:regrt}.
Furthermore we have that
$$\REG^n \geq_\W \lambda\regHT^{=n}[\ap] \geq_\compred \HT^{=n}_k[\ap],$$
from Theorem \ref{thm:acareght} and Proposition~\ref{prop:regHTtoHTeqn}.

Moreover, whereas $\lambda\regHT^{=n}[\ap] \geq_\W \RAN$ for any $n\geq 2$ (Theorem~\ref{thm:regtoaca}), we have
that $\HT^{=n}_k[\ap] \geq_\W \RAN$ only for $n\geq 3$ and $k\geq 2$ (by an easy adaptation of the proof of Theorem 3.3 in~\cite{Car-Kol-Lep-Zda:20}).
Also note that $\RT^n_k \geq_\sW \HT^{=n}_k[\ap]$ by a straightforward reduction (see~\cite{Car-Kol-Lep-Zda:20}).

Some non-reducibility results can be gleaned from the above and known non-reducibility results from the literature. First, Dorais, Dzhafarov, Hirst, Mileti, and Shafer showed that $\RT^n_k \not\leq_{\sW} \RT^n_j$ (Theorem 3.1 of~\cite{DJSW:16}). Then $\RT^n_k \not\leq_{\W} \RT^n_j$ was proved by Brattka and Rakotoniaina \cite{Bra-Rat:17} and, independently, by Hirschfeldt and Jockusch~\cite{Hir-Joc:16}. Patey further improved this result by showing that the computable reduction does not hold either~\cite{Pat:16}; i.e., $\RT^n_k \not\leq_{\compred} \RT^n_j$ for all $n\geq 2$, $k > j \geq 2$. We can derive, among others, the following corollaries. 

\begin{corollary}\label{prop:reg-forallk-rt}
For each $n, k \geq 2$, $\REG^n \not\leq_{\compred} \RT^n_k$.
\end{corollary}

\begin{proof}
From Proposition~\ref{prop:regrt} we know that $\RT^n_{k+1} \leq_{\compred} \REG^n$, so if we had $\REG^n \leq_{\compred} \RT^n_k$ we could transitively obtain $\RT^n_{k+1} \leq_{\compred} \RT^n_k$, hence contradicting the fact that $\RT^n_{k+1} \not\leq_{\compred} \RT^n_k$ proved by Patey~\cite{Pat:16}.
\end{proof}

\begin{corollary}
$\RT^3_3\not\leq_{\compred}\lambda\regHT^{= 2}[\ap]$.
\end{corollary}

\begin{proof}
It is known from \cite{Pat:16} that $\RT^3_3 \not\leq_{\compred} \RT^3_2$. 
On the other hand $\lambda\regHT^{= 2}[\ap] \leq_{\mathrm{W}} \RT^3_2$, since
$\lambda\regHT^{= 2}[\ap] \leq_{\mathrm{W}} \REG^2$ (Theorem \ref{thm:acareght}) and 
$\REG^2 \leq_\sW  \RT^3_2$ (from the proof of Theorem 6.14 in \cite{Hir:phd}) and since the involved
reducibilities satisfy the following inclusions and are transitive: $\leq_\sW \subseteq \leq_\W \subseteq \leq_\compred$.
\end{proof}

As proved in~\cite{Car-Kol-Lep-Zda:20}, restrictions of Hindman's Theorem have intriguing connections
with the so-called 
Increasing Polarized Ramsey's Theorem for pairs $\IPT^2_2$ of Dzhafarov and Hirst~\cite{Dza-Hir:11}. For example, 
$\HT^{=2}_2[\ap] \geq_\W \IPT^2_2$ (Theorem 4.2 in \cite{Car-Kol-Lep-Zda:20}). By this result and Proposition~\ref{prop:regHTtoHTeqn} we have the following corollary.

\begin{corollary}\label{coro:regHTtoIPT}
$\IPT^2_2\leq_\mathrm{\compred} \lambda\regHT^{=2}[\ap]$.
\end{corollary}

Note that $\IPT^2_2$ is the strongest known lower bound for $\HT^{=2}_2[\ap]$ in terms of reductions. Some interesting lower bounds on $\HT^{=2}$ \emph{without} apartness are in~\cite{Csi-Dzh-Hir-Joc-Sol-Wes:19}. We haven't investigated $\lambda\regHT^{=n}$ \emph{without} the apartness condition; we conjecture that the lower bounds on $\HT^{=2}$ (\emph{without} apartness) from \cite{Csi-Dzh-Hir-Joc-Sol-Wes:19} can be adapted to $\lambda\regHT^{=2}$.

\subsection{Bounded regressive Hindman's Theorem and Well-ordering Principles}

Let $(\ordX, <_\ordX)$ be a linear ordering. We denote by $\omega^\ordX$ 
the collection of finite sequences of the form $(x_1, x_2, \dots, x_s)$
such that, for all $i \in [1, s]$, $x_i \in \ordX$ and, for all $i,j \in [1,s]$ 
such that $i < j$, $x_i \geq_\ordX x_j$. We call the $x_i$s the \emph{components} of $\sigma$.
We denote by $|\sigma|$ the \emph{length of} $\sigma$, i.e. $|\sigma| = s$. 
We order $\omega^\ordX$ lexicographically. Then, if $\sigma, \tau \in \omega^\ordX$ 
and $\sigma$ strictly extends $\tau$, we have $\sigma > \tau$. If $j$ is least such that
$x_j  = \sigma(j) \neq \tau(j) = x'_j$ and $x_j >_\ordX x_j'$ then $\sigma > \tau$. Otherwise $\tau \geq \sigma$.

If $(\ordX, <_\ordX)$ is a well-ordering, then the just defined ordering on $\omega^\ordX$ is also a well-ordering (provably in sufficiently strong theories). In this case we can then identify an element $\sigma = (x_1, x_2, \dots, x_s)$ of $\omega^\ordX$
with the ordinal $\omega^{x_1}+ \omega^{x_2} + \dots + \omega^{x_s}$. 
The lexicographic ordering of $\omega^\ordX$ coincides with the usual ordering of ordinals in Cantor Normal Form. 

The well-ordering preservation principle (or well-ordering principle) for base-$\omega$ exponentiation is the following $\Pi^1_2$-principle: 
$$ \forall \ordX (\WO(\ordX) \to \WO(\omega^\ordX)),$$
where $\WO(Y)$ is the standard $\Pi^1_1$-formula stating that $Y$ is a well-ordering. We abbreviate the above
well-ordering preservation principle by $\WOP(\ordX \mapsto \omega^\ordX)$.

It is known that $\WOP(\ordX \mapsto \omega^\ordX)$ is equivalent to $\ACA_0$ by results of Girard and Hirst
(see \cite{Hir:94}). 
A direct combinatorial proof from $\RT^3_3$ to $\WOP(\ordX \mapsto \omega^\ordX)$ in 
$\RCA_0$ was given by Carlucci and Zdanowski~\cite{Car-Zda:12} (the proof yields a Weihrauch reduction as clear by inspection). On the other hand, we proved in 
Theorem~\ref{thm:regtoaca} that, for any $n \geq 2$, $\lambda\regHT^{= n}[\ap]$ implies $\ACA_0$ over $\RCA_0$. Therefore in $\RCA_0$
we have that, for $n \geq 2$, $\lambda\regHT^{= n}$ with apartness implies $\WOP(\ordX\mapsto \omega^\ordX)$. 
However, we can not use the same arguments to derive an analogous chain of reductions. In the next theorem we show 
that $\WOP(\ordX\mapsto \omega^\ordX)$ is Weihrauch-reducible to $\lambda\regHT^{= n}[\ap]$, 
while also giving a direct proof of the implication in $\RCA_0$. This result relates for the first time, to the best of our knowledge, Hindman-type theorems and transfinite well-orderings.

To make the principle $\WOP(\ordX\mapsto \omega^\ordX)$ amenable to questions of reducibility it is natural to consider its 
contrapositive form: an {\em instance} is an infinite descending sequence in $\omega^\ordX$ and a {\em solution} is an infinite descending sequence in $\ordX$ (in fact, one might require that the solution consists of terms already occurring as subterms of the elements of the instance sequence, as is the case in our argument below).

We briefly describe the idea in the proof of Theorem \ref{thm:reghtwop} below. Let $\ordX$ be a linear ordering. Let $\alpha=(\alpha_i)_{i \in \Nat}$ be an infinite decreasing sequence in $\omega^\ordX$. We show, using $\lambda\regHT^{= 2}[\ap]$, that there exists an infinite decreasing sequence in $\ordX$. The proof uses ideas from our proof of the fact that $\lambda\regHT^{= 2}$ with apartness implies $\ACA_0$ (Theorem \ref{thm:regtoaca}) adapted to the present context, based on the following analogy between deciding the first Turing jump $\emptyset'$ and computing an infinite descending sequence in $\ordX$. Given an enumeration of $\emptyset'$ and a number $n$, $\RCA_0$ knows that there is a $b$ such that all numbers in $\emptyset'$ below $n$ appear within $b$ steps of 
the enumeration, but is not able to compute this $b$. Similarly, given an ordinal $\alpha$ in an infinite decreasing sequence in $\omega^\ordX$, $\RCA_0$ knows that there is a $b$ such that if a term of $\alpha$ ever decreases, it will do so 
by the $b$-th term of the infinite descending sequence, but is unable to compute such a $b$. More precisely, 
while one can computably run through the given infinite descending sequence to find the first point at which an exponent 
of a component of $\alpha$ is decreased, we can not locate computably the leftmost such component. An appropriately 
designed colouring will ensure that the information about such a $b$ can be read off from the elements of any apart solution to $\lambda\regHT^{= n}$. 

We start with the following simple Lemma. For technical convenience in the rest of this section we index infinite sequences and sets starting from $1$.

\begin{lemma}\label{claim1} The following is provable in $\RCA_0$: If $\alpha=(\alpha_i)_{i\in\Nat^+}$ is an infinite descending sequence in $\omega^\ordX$, then
$$\forall n\ \exists n'\ \exists m\leq |\alpha_n| \ \big(n'>n \ \land m \leq |\alpha_{n'}| \ \land \alpha_{n,m} >_{\ordX} \alpha_{n',m}\big),$$
where $\alpha_{i,j}$ denotes the $j$-th component of $\alpha_i$ for $j \in [1,|\alpha_i|]$ and is otherwise undefined.
\end{lemma}

\begin{proof}

Assume by way of contradiction that the statement is false, as witnessed by $n$, and recall that for any distinct $\sigma, \tau \in \omega^\ordX$, we have $\sigma < \tau$ if and only if either (1.) $\sigma$ is a proper initial segment of $\tau$, or (2.) there exists $m$ such that $\sigma(m) <_\ordX \tau(m)$ and  $\sigma(m') = \tau(m')$ for each $m'<m$. Then we can show that:
$$
\forall p\ (p \geq n \rightarrow (\alpha_{p+1} \text{ is a proper initial segment of both } \alpha_p \text{ and } \alpha_n))
$$
by $\Delta^0_1$-induction.

The case $p=n$ is trivial, since $\alpha_n >_{\ordX} \alpha_{n+1}$ and (2.) cannot hold by assumption.

For $p>n$, by induction hypothesis we know that $\alpha_p$ is a proper initial segment of $\alpha_n$. 
Since $\alpha_{p+1} <_{\ordX} \alpha_p$, $\ \alpha_{p+1}$ must be a proper initial segment of $\alpha_p$, otherwise the leftmost component differing between $\alpha_{p+1}$ and $\alpha_p$ -- i.e. the component of $\alpha_{p+1}$ with index $m$ witnessing (2.) -- would contradict our assumption, for we would have $m\leq |\alpha_p|$ and $\alpha_{p+1,m} <_\ordX \alpha_{p,m} = \alpha_{n,m}$.

So  $\alpha_{p+1}$ must be a proper initial segment of $\alpha_p$ and, by our assumption, it must be a proper initial segment of $\alpha_n$ as well.

The previous statement implies that:
$$
\forall p\ (p \geq n \rightarrow |\alpha_p| > |\alpha_{p+1}|),
$$
which contradicts $\WO(\omega)$. This concludes the proof.
\end{proof}

\begin{theorem}\label{thm:reghtwop}
Let $n \geq 2$. $\lambda\regHT^{= n}[\ap]$ implies $\WOP(\ordX\mapsto \omega^\ordX)$ over $\RCA_0$.
Moreover, $\lambda\regHT^{=n}[\ap]\geq_\W \WOP(\ordX\mapsto \omega^\ordX)$.
\end{theorem}

\begin{proof}
Let $\alpha= (\alpha_n)_{n\in \Nat^+}$ be an infinite descending sequence in $\omega^\ordX$. 
We say that $\alpha_{n,m}$ is \emph{decreasible} if there exists a $n'>n$ such that $\alpha_{n',m} <_\ordX \alpha_{n,m}$. In this case we say that $\alpha_{n',m}$ \emph{decreases} $\alpha_{n,m}$. 
With this terminology Lemma \ref{claim1} says that $\RCA_0$ knows that for all $i\geq 1$ there exists $j \in [1, |\alpha_i|]$ such that $\alpha_{i,j}$ is decreasible. If $\alpha_{n',m}$ decreases $\alpha_{n,m}$ and no $\alpha_{k,m}$ 
with $k < n'$ decreases $\alpha_{n,m}$ we call $\alpha_{n',m}$ the \emph{least decreaser} of $\alpha_{n,m}$. 

Now suppose that $f:\Nat \to \Nat$ is a function with the following property:

\bigskip
{\em Property P}: For all $i\in \Nat^+$ for all $j \in [1, |\alpha_i|]$
if $\alpha_{i,j}$ is decreasible then $\alpha_{i,j}$ is decreased by $\alpha_{k,j}$ for some $k\leq f(i)$. 

\bigskip
We first show that given such an $f$ we can compute (in $f$ and $\alpha$) 
an infinite descending sequence $(\sigma_i)_{i\in\Nat^+}$ in $\mathcal{X}$ as follows. 

\bigskip
{\bf Step $1$}.  
Pick the leftmost decreasible component of $\alpha_1$ (which exists by Lemma \ref{claim1}). 
This can be done by inspecting all components 
in $\alpha$ up through $\alpha_{f(1)}$, since $f$ has Property $P$. 

Let $p_1$ be the position of the leftmost decreasible component of $\alpha_1$. 
Pick the smallest $d_1 \leq f(1)$ such that $\alpha_{d_1, p_1}$ decreases 
$\alpha_{1,p_1}$. 
We set $\sigma_1 = \alpha_{d_1, p_1}$ and observe that
all decreasible components in $\alpha_{d_1}$ occur at positions $\geq p_1$. 
Suppose otherwise and let $1\leq p^* < p_1$ be such that $\alpha_{d_1, p^*}$ is decreasible. 
Let $d^* > d_1$ such that $\alpha_{d^*, p^*}$ decreases $\alpha_{d_1,p^*}$. Then 
$\alpha_{d^*, p^*} <_{\ordX} \alpha_{d_1,p^*}$ by definition of decreasible. 
On the other hand, by choice of $d_1$ and $p_1$, and since
$p^* < p_1$, it must be the case that $\alpha_{d_1, p^*} = \alpha_{1, p^*}$. Hence $\alpha_{1, p^*}$ is a decreasible
component in $\alpha_1$ on the left of position $p_1$, which contradicts the choice of $p_1$. 

\bigskip
{\bf Step $i+1$ ($i > 0$)}. Suppose $d_i, p_i, \sigma_i$ are defined so that $\sigma_i = \alpha_{d_i, p_i}$, $(\sigma_j)_{1\leq j\leq i}$ is decreasing in $\mathcal{X}$ and all decreasible components in $\alpha_{d_i}$ occur at positions $\geq p_i$. 

Pick the leftmost decreasible component in $\alpha_{d_i}$ (which exists by Lemma \ref{claim1}). This can be done by inspecting all components in $\alpha$ up to 
$\alpha_{f(d_i)}$, since $f$ has Property $P$.  Let $\alpha_{d_i, \ell}$ be the chosen component. Set $p_{i+1}=\ell$ and note that necessarily $p_{i+1}\geq p_i$.


Pick $d \leq f(d_i)$ minimal such that $\alpha_{d, p_{i+1}}$
decreases $\alpha_{d_i, p_{i+1}}$. Set $d_{i+1}=d$. 
Let $\sigma_{i+1} = \alpha_{d_{i+1}, p_{i+1}}$. 
Obviously $\sigma_i >_\mathcal{X} \sigma_{i+1}$, since $\sigma_i  = \alpha_{d_i,p_i} \geq \alpha_{d_i,p_{i+1}} >_\mathcal{X} \alpha_{d_{i+1}, p_{i+1}} = \sigma_{i+1}$ (note that $p_i \leq p_{i+1}$).

We observe that also the last part of the inductive invariant is guaranteed, since no decreasible component 
in $\alpha_{d_{i+1}}$ occurs on the left of $p_{i+1}$. 
Suppose otherwise as witnessed by $1\leq p^* < p_{i+1}$. Let $d^* > d_{i+1}$ such that $\alpha_{d^*,p^*}$
decreases $\alpha_{d_{i+1}, p^*}$. Then $\alpha_{d^*, p^*}$ also decreases $\alpha_{d_i, p^*}$ 
since $\alpha_{d_{i+1}, p^*}= \alpha_{d_{i}, p^*}$, where the latter is due to the fact that $\alpha$ is decreasing 
and $p^*$ is less than $p_{i+1}$, which is the position of the leftmost decreasible component in $\alpha_{d_i}$. This contradicts the choice of $p_{i+1}$. 

\bigskip
We next show how to obtain a function satisfying Property $P$ from a solution of $\lambda\regHT^{=n}[\ap]$ for a suitable colouring. The argument is similar to the proof of Theorem \ref{thm:regtoaca}.

For this purpose it is convenient to use a sequence $\beta$ of all the components of the terms $\alpha_n$ in $\alpha$, enumerated in order of appearance: more precisely, $(\beta_h)_{h\in\Nat^+}$ is the ordered sequence $\alpha_{1,1}, \alpha_{1,2}, \dots, \alpha_{1,|\alpha_1|}, \alpha_{2,1}, \alpha_{2,2}, \dots, \alpha_{2,|\alpha_2|}, \dots$. 
This sequence is obviously easily computable from $\alpha$. 
Formally we construct such a sequence by first defining a function $\iota : \Nat^+\times\Nat^+\to \Nat^+$ as follows: $\iota(n,m)= \sum_{1\leq k < n} |\alpha_k| + m$, for all $n\in\Nat^+$ and all $m \in [1, |\alpha_n|]$, while 
$\iota(n,m)=0$ in all other cases. 
We correspondigly fix functions $t: \Nat^+ \to \Nat^+$ and $p:\Nat^+\to\Nat^+$ such that for each $n\in\Nat^+$ we have
$\iota(t(n),p(n))=n$. The sequence $(\beta_h)_{h\in\Nat^+}$ of all components appearing in $\alpha$ is then defined
by setting $\beta_h = \alpha_{t(h),p(h)}$.

Define $c:\Nat\to \Nat$ as follows: $c(x)=$ the unique $i < \lambda(x)$ satisfying the following conditions:
\begin{enumerate}
\item There exists $j$ such that $\lambda(x) \leq j < \mu(x)$ and $\beta_j$ is the least decreaser of $\beta_i$, and 
\item For all $j'$ such that $j < j'  < \mu(x)$, if $\beta_{j'}$ is the least decreaser of $\beta_{i'}$ then $i' \geq \lambda(x)$. 
\end{enumerate}
If no such $i$ exists, we set $c(x)=0$.

The function $c$ is computable in $\alpha$ and $\lambda$-regressive. Let $H = \{ h_1 < h_2 < h_3 < \dots \}$ be an apart solution to $\lambda\regHT^{= n}$ for $c$. The following Claim ensures the existence of an $(\alpha \oplus H)$-computable function with Property $P$.

\begin{claim}\label{claim2} For each $h_k\in H$ and each $\alpha_{\ell,m}$ such that $\iota(\ell,m) < \lambda(h_k)$, if there exists $\alpha_{\ell',m}$ such that $\alpha_{\ell',m}$ decreases $\alpha_{\ell,m}$ then there exists such an $\alpha_{\ell',m}$ with $\iota(\ell',m)<\mu(h_{k+n-1})$.
\end{claim}

\begin{proof}[Proof of Claim \ref{claim2}]
Assume by way of contradiction that there is some $h_k \in H$ and some $\alpha_{\ell,m}$ with $\iota(\ell,m)=i < \lambda(h_k)$ such that $\alpha_{\ell,m}$ is decreasible but not by any $\alpha_{\ell',m}$ with $\iota(\ell',m) < \mu(h_{k+n-1})$. 

Let $b$ be such that if $\alpha_{\ell'',m}$ is decreasible and $\iota(\ell'',m) < \lambda(h_k)$, then there exists $\ell',m$ such that $\iota(\ell',m) < b$ and $\alpha_{\ell',m}$ decreases $\alpha_{\ell'',m}$. The existence of $b$ can be proved in $\RCA_0$ using the following instance of strong $\Sigma^0_1$-bounding (similarly as in the proof of Theorem \ref{thm:regtoaca}):
$$ \forall n \exists b \forall i < n ( \exists j (\alpha_{t(j),p(j)} \mbox{ decreases } \alpha_{t(i),p(i)}) \to \exists j < b (\alpha_{t(j),p(j)} \mbox{ decreases } \alpha_{t(i),p(i)}).$$
Since $H$ is infinite, there is an $h_{k'} \in H$ such that $h_{k'} > h_{k+n-1}$ and $\mu(h_{k'}) \geq b$. Then, by min-term-homogeneity, $c(h_k + \cdots + h_{k+n-1}) = c(h_k + \cdots + h_{k+n-2} + h_{k'})$. But by choice of $h_k$, $h_{k'}$ and the definition of $c$, we can show that $c(h_k + \cdots + h_{k+n-1}) \neq c(h_k + \cdots + h_{k+n-2} + h_{k'})$, 
yielding a contradiction.

To see this we reason as follows. First observe that, by apartness of $H$, the following identities hold:
$$\lambda(h_k + \cdots + h_{k+n-1}) = \lambda(h_k + \cdots + h_{k+n-2} + h_{k'}) = \lambda(h_k),$$
and
$$\mu(h_k + \cdots + h_{k+n-2} + h_{k'})=\mu(h_{k'}).$$
Let $j \in [\mu(h_{k+n-1}), \mu(h_{k'}))$ be such that $\alpha_{t(j),p(j)}$ is the least decreaser of $\alpha_{\ell,m}$. Such a $j$ exists by choice of $\alpha_{\ell,m}$, $h_k$ and $h_{k'}$. In fact, by hypothesis, 
$\alpha_{\ell,m}$ is decreasible but not by any component with $\iota$-index below $\mu(h_{k+n-1})$. By choice
of $h_k'$ the least decreaser of $\alpha_{\ell,m}$ must have $\iota$-index smaller than $\mu(h_{k'})$, since $\iota(\ell,m) < \lambda(h_k)$.

First note that
$c(h_k + \cdots + h_{k+n-2} + h_{k'})$ cannot be $0$, 
since this occurs if and only if there is no $i^* < \lambda(h_k)$ such that for some $j^* \in [\lambda(h_k), \mu(h_{k'}))$, $\alpha_{t(j),p(j)}$ decreases $\alpha_{t(i^*), p(i^*)}$; but the latter is false by choice of $h_k$ and $h_{k'}$. 

If $c(h_k + \cdots + h_{k+n-1})$ takes some non-zero value $i^* < \lambda(h_k)$, then this same value cannot
be taken by $c(h_k + \cdots + h_{k+n-2} + h_{k'})$ under our assumptions. 
If it were, it would mean that $\alpha_{t(i^*),p(i^*)}$ is decreased for the 
first time by some $\alpha_{t(j^*), p(j^*)}$ with $j^* < \mu(h_{k'})$ such that $j^*$ is also maximal below $\mu(h_{k'})$ such that $\alpha_{t(j^*), p(j^*)}$ is the least decreaser of some $\alpha_{t(q),p(q)}$ with  $q < \lambda(h_k)$. This is impossible since the least decreaser of $\alpha_{t(i^*),p(i^*)}$, by the hypothesis that $c(h_k + \cdots + h_{k+n-1})=i^*$, occurs earlier in the sequence of the $\beta_h$'s than the least decreaser of $\alpha_{t(i),p(i)}$ since, by the definition of $c$, it must be that $j^* < \mu(h_k + \cdots + h_{k+n-1})$ and the latter value, by apartness, equals $\mu(h_{k+n-1})$, as noted above. On the other hand, $j$ is in $[\mu(h_{k+n-1}), \mu(h_{k'}))$, so that $j^*<j$. Thus $j^*$ cannot be maximal below $\mu(h_{k'})$ such that $\alpha_{t(j^*),p(j^*)}$ is the least decreaser of some $\alpha_{\ell'',m}$ with $\iota(\ell'',m)$ below $\lambda(h_k)$, as required by the definition of $c$, since $\alpha_{t(j),p(j)}$ is such a least decreaser of $\alpha_{t(i),p(i)}$, and $i<\lambda(h_k)$. 

This proves the Claim.
\end{proof}

Now it is sufficient to observe that the $(\alpha \oplus H)$-computable function 
$f$ defined as follows has the Property $P$: on input $n$, 
pick the least $k$ such that $\sum_{1\leq n' \leq n} |\alpha_{n'}| < \lambda(h_k)$ and 
let $f(n)$ be the $\alpha$-index of the $\mu(h_{k+n-1})$-th 
element in the sequence $\beta$ of all components appearing in $\alpha$, i.e., $f(n) = t(\mu(h_{k+n-1}))$. 
That this choice of $f$ satisfies Property P is implied by Claim \ref{claim2} above. This concludes the proof of the theorem.
\end{proof}

The proof of Proposition 1 in \cite{Car-Zda:12} shows that $\WOP(\ordX\to \omega^\ordX) \leq_\W \RT^3_3$. 
The proof of Theorem \ref{thm:reghtwop} can be adapted to show that $\WOP(\ordX\to \omega^\ordX) \leq_\W \HT^{=3}_2[\ap]$. Details will be reported elsewhere.

The main reductions between restrictions of $\HT$, restrictions of $\lambda\regHT$ and other principles of interest are visualized in 
Figure \ref{fig:reductions}.

\begin{figure}
\centering
\[\begin{tikzcd}[row sep=3em,column sep=2em]
& \lambda\regHT^{\leq n}[\ap] \arrow[d, "\sW"] \arrow[rr, "\compred"] & & \HT^{\leq n}_k[\ap]  \arrow[d, "\sW"] & \\
\REG^n \arrow[r, "\sW"] \arrow[rrrr, bend right=65, "\compred" below] & \lambda\regHT^{= n}[\ap] \arrow[d, "\W"] \arrow[rr, "\compred"] \arrow[dr, "\W (n\geq 2)"]  & & \HT^{= n}_k[\ap] \arrow[dl, "\W (n\geq 3)"]  
&  \arrow[l, "\sW" above] \RT^n_k  \\
& \WOP(\ordX \to \omega^\ordX) & \RAN & & 
\end{tikzcd}\]
\caption{Diagram of reductions. $\HT^{\leq n}[\ap] \leq_\compred \lambda\regHT^{\leq n}[\ap]$ is from Proposition \ref{prop:regHTtoHTeqn}.
That the versions with sums of exactly $n$ terms reduce to the corresponding versions for sums of $\leq n$ terms is a trivial observation. The reduction $\WOP(\ordX\to \omega^\ordX)\leq_\W \lambda\regHT^{=n}$ for $n\geq 2$ is Theorem \ref{thm:reghtwop}. The reduction $\RAN \leq_\W \lambda\regHT^{=n}$ for $n\geq 2$ is Theorem \ref{thm:regtoaca}. The reduction $\RAN \leq_\W \HT^{=n}_k[\ap]$ for $n\geq 3, k\geq 2$ is from \cite{Car-Kol-Lep-Zda:20}. The reduction $\RT^n_k \leq_\compred \REG^n$ is from Proposition \ref{prop:regrt}. The reduction $\HT^{=n}_h \leq_\sW \RT^n_k$ is folklore.} \label{fig:reductions}
\end{figure}
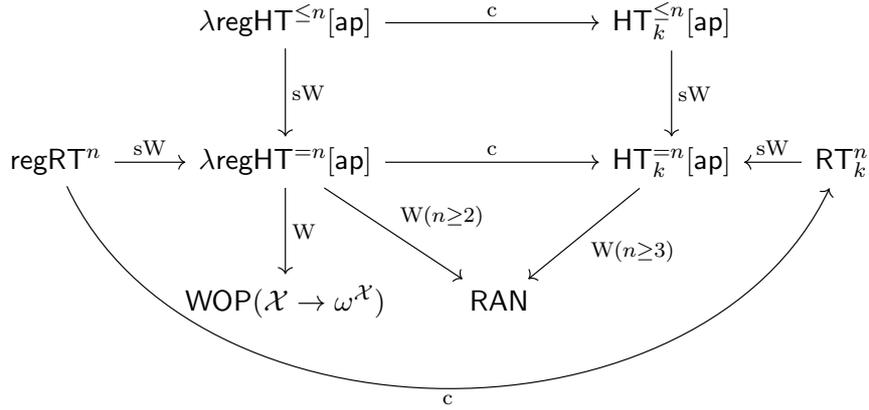

\section{Conclusion and open questions}

In analogy with Kanamori-McAloon's Regressive Ramsey's Theorem \cite{Kan-McA:87} we obtained a Regressive Hindman's Theorem 
as a straightforward corollary of Taylor's Canonical Hindman's Theorem \cite{Tay:76} restricted to a suitable class of regressive functions
and relative to an appropriate variant of min-homogeneity. We studied the strength of this principle and of its 
restrictions in terms of provability over $\RCA_0$ and computable reductions. 

In particular we showed that the seemingly weakest (non-trivial) restriction of our Regressive Hindman's Theorem ($\lambda\regHT^{=2}$), with a natural apartness condition on the solution set, is equivalent
to $\ACA_0$. This restriction ensures that sums of two numbers from the solution set get the same colour if they have the same minimum term. For the restrictions of the standard Hindman's Theorem to sums of exactly $n$ elements, the level of $\ACA_0$ is reached only when we consider sums of exactly $3$ elements. This situation is analogous to that of $\REG^2$ 
when compared to $\RT^3_2$. Furthermore, we proved that the well-ordering preservation principle that characterizes $\ACA_0$ ($\WOP(\ordX \to \omega^\ordX)$) is Weihrauch-reducible to $\lambda\regHT^{=2}$ with apartness.

Many open questions remain concerning the strength of the Regressive Hindman's Theorem, of its restrictions, and of related principles. Here are some natural ones. 

\begin{question}
What are the optimal upper bounds for $\canHT$, for $\lambda\regHT$ and for $\lambda\regHT^{\leq n}$?
\end{question}

\begin{question}
Is $\lambda\regHT$ implied by/reducible to $\HT$ (and similarly for bounded versions)?
\end{question}

\begin{question}
What is the strength of $\lambda\regHT^{=2}$ without apartness? More generally, how do the bounded Regressive Hindman's Theorems behave with respect to apartness?
\end{question}

\begin{question}
Can the reductions in Proposition \ref{prop:regHTtoHTeqn} and Theorem \ref{thm:reghtwop} be improved to stronger reductions?
\end{question}

Very recently, Hirschfeldt and Reitzes \cite{Hir-Rei:pre} investigated Hindman-type variants of the Thin Set Theorem which, as is the case
for our Regressive Hindman's Theorem, deals with colourings with unboundedly many colours. It would be interesting to 
investigate possible relations between the two families. 

\bigskip
{\bf Acknowledgments} We thank the anonymous referees for pointing out some inaccuracies in previous versions of the paper. 
Their further suggestions led to a significant improvement of the presentation, to the addition of Theorem \ref{thm:canap}.


\begin{thebibliography}{1}

\bibitem{Bla:05}
A.~Blass.
Some questions arising from Hindman's Theorem.
{\em Scientiae Mathematicae Japonicae}, 62, 331--334, 2005.

\bibitem{Bla-Hir-Sim:87}
A.~R.~Blass, J.L.~Hirst, S.~G.~Simpson. 
Logical analysis of some theorems of combinatorics and topological dynamics. 
{\em Logic and combinatorics}, Contemporary Mathematics, American Mathematical Society, 65, 125--156, 1987.

\bibitem{Bra-Rat:17}
V.~Brattka, T.~Rakotoniaina.
On the uniform computational content of Ramsey’s Theorem. 
{\em The Journal of Symbolic Logic}, 4, 1278--1316, 2017.


\bibitem{Car:16:wys}
L.~Carlucci. 
Weak Yet Strong restrictions of Hindman's Finite Sums Theorem. 
{\em Proceedings of the American Mathematical Society}, 146, 819--829, 2018.

\bibitem{Car-Kol-Lep-Zda:20}
L.~Carlucci, L.~A.~Ko{\l}odziejczyk, F.~Lepore, K.~Zdanowski.
New bounds on the strength of some restrictions of Hindman's Theorem.
{\em Computability}, 9, 139--153, 2020.

\bibitem{Car:2021}
L.~Carlucci. 
Restrictions of Hindman’s Theorem: an overview.
{\em Connecting with Computability: CiE 2021}, Lecture
Notes in Computer Science, 12813, Springer, 94--105, 2021.

\bibitem{Car-Zda:12}
L.~Carlucci, K.~Zdanowski, 
A note on Ramsey Theorems and Turing Jumps. {\em How the World Computes. CiE 2012}, 
Lecture Notes in Computer Science, 7318, 89--95, 2012.

\bibitem{Csi-Dzh-Hir-Joc-Sol-Wes:19}
B.~F.~Csima, D.~D.~Dzhafarov, D.~R.~Hirschfeldt, 
C.~G.~Jockusch., R.~Solomon, L.~B.~Westrick. 
The Reverse Mathematics of Hindman's Theorem for sums of exactly two elements. 
{\em Computability}, 8, 253--263, 2019.

\bibitem{DDHMS:16}
F.~G.~Dorais, D.~Dzhafarov, J.~L.~Hirst, J.~P.~Mileti, P.~Shafer. 
On uniform relationships between combinatorial problems. 
{\em Transactions of the American Mathematical Society}, 368 (2), 1321--1359, 2016.

\bibitem{Dza-Hir:11}
D.~D.~Dzhafarov, J.~L.~Hirst.
The polarized Ramsey's theorem.
{\em Archive for Mathematical Logic}, 48(2), 141--157, 2011.

\bibitem{DJSW:16}
D.~D.~Dzhafarov, C.~G.~Jockusch, R.~Solomon, L.~B.~Westrick. 
Effectiveness of Hindman's Theorem for bounded sums. 
{\em Proceedings of the International Symposium on Computability and Complexity} 
(in honour of Rod Downey's 60th birthday), Lecture Notes in Computer Science, 10010, 
134--142, Springer, 2016.

\bibitem{Dza-Mum:22}
D.~D.~Dzhafarov, C.~Mummert.
{\em Reverse Mathematics. Problems, Reductions, and Proofs}.
Theory and Applications of Computability, Springer Cham, 2022.


\bibitem{Erd-Rad:50}
P.~Erd\H{o}s, R.~Rado.
A combinatorial theorem.
{\em Journal of the London Mathematical Society} 25, 249--255, 1950.

\bibitem{Hin:72}
N.~Hindman. 
The existence of certain ultrafilters on $N$ and a conjecture of Graham and Rothschild.
{\em Proceedings of the American Mathematical Society}, 36 (2), 341--346, 1972.

\bibitem{Hin:74}
N.~Hindman. 
Finite sums from sequences within cells of a partition of N.
{\em Journal of Combinatorial Theory Series A}, 17, 1--11, 1974.

\bibitem{Hin-Lea-Str:03}
N.~Hindman, I.~Leader, D.~Strauss. 
Open problems in partition regularity.
{\em Combinatorics Probability and Computing}, 12, 571--583, 2003.

\bibitem{Hir:STT:14}
D.~R.~Hirschfeldt.
{\em Slicing the Truth (On the Computable and Reverse Mathematics of Combinatorial Principles)}. 
Lecture Notes Series, 28, Institute for Mathematical Sciences, National University of Singapore, 2014. 

\bibitem{Hir-Joc:16}
D.~R.~Hirschfeldt, C.~G.~Jockusch.
On notions of computability-theoretic reduction between $\Pi^1_2$ principles. 
{\em Journal of Mathematical Logic} 16 (1), 1278--1316, 2016.

\bibitem{Hir-Rei:pre}
D.~R.~Hirschfeldt, S.~C.~Reitzes.
Thin Set Versions of Hindman's Theorem.
{\tt 	arXiv:2203.08658}.

\bibitem{Hir:phd}
J.~L.~Hirst.
Combinatorics in Subsystems of Second Order Arithmetic.
PhD thesis, The Pennsylvania State University, 1987.

\bibitem{Hir:94}
J.~L.~Hirst.
Reverse mathematics and ordinal exponentiation,
{\em Annals of Pure and Applied Logic}, 66 (1), 1--18, 1994.

\bibitem{Hir:12:HvsH}
J.~L.~Hirst.
Hilbert vs.~Hindman. 
{\em Archive for Mathematical Logic}, 51 (1-2), 123--125, 2012. 

\bibitem{Hir:14}
J.~L.~Hirst. 
Disguising induction: proofs of the pigeonhole principle for trees.
{\em Foundational adventures}, Tributes, 22, Coll Publ., 113--123, 2014.


\bibitem{Joc:72}
\newblock C.~G.~Jockusch.
\newblock Ramsey's theorem and recursion theory, 
\newblock {\em Journal of Symbolic Logic}, 37, 268--280, 1972.


\bibitem{Kan-McA:87} 
A.~Kanamori, K.~McAloon.
On G\"odel incompleteness and finite combinatorics. 
\textit{Annals of Pure and Applied Logic}, 33 (1), 23--41, 1987. 


\bibitem{Mil:05} J.~Mileti.
Partition theorems and computability theory. 
\textit{Bulletin of Symbolic Logic}, 11 (3), 411--427, 2005.

\bibitem{Mil:08} J.~Mileti. 
The canonical Ramsey's Theorem and computability theory. 
\textit{Transactions of the American Mathematical Society}, 160, 1309--1340, 2008.

\bibitem{Mon:11:open}
A.~Montalb\'{a}n.
Open questions in Reverse Mathematics.
{\em Bulletin of Symbolic Logic}, 17 (3), 431--454, 2011.

\bibitem{Pat:16}
L.~Patey. 
The weakness of being cohesive, thin or free in reverse mathematics. 
{\em Israel Journal of Mathematics}, 216, 905--955, 2016.

\bibitem{Rat:wop:pre}
M.~Rathjen.
Well-ordering principles in Proof Theory and Reverse Mathematics. 
Preprint, \texttt{arXiv:2010.12453}.


\bibitem{Sim:SOSOA}
S.~Simpson.
{\em Subsystems of Second Order Arithmetic}.
Second Edition, Cambridge University Press, New York, NY, Association for Symbolic Logic, 2009.

\bibitem{Tay:76}
A.~D.~Taylor.
A canonical partition relation for finite subsets of $\omega$. 
{\em Journal of Combinatorial Theory (A)}, 21, 137--146, 1976.




\end{thebibliography}
\end{document}